\documentclass[12pt]{article}
\usepackage{amssymb,amsmath,amsthm, amsfonts}
\usepackage{graphicx}
\usepackage{epsfig}
\usepackage{tikz}

\def\disp{\displaystyle}

\def\dref#1{(\ref{#1})}

\theoremstyle{plain}
\newtheorem{theorem}{Theorem}[section]
\newtheorem{lemma}{Lemma}[section]

\newtheorem{corollary}{Corollary}[section]

\theoremstyle{definition}
\newtheorem{definition}{Definition}[section]

\newtheorem{remark}{Remark}[section]

\numberwithin{equation}{section}

\begin{document}

\title{\bf A new result for boundedness of solutions to a quasilinear higher-dimensional chemotaxis--haptotaxis model
 with nonlinear diffusion}

\author{
Jiashan Zheng 
\thanks{Corresponding author.   E-mail address:
 zhengjiashan2008@163.com (J.Zheng)}
\\
    School of Mathematics and Statistics Science,\\
     Ludong University, Yantai 264025,  P.R.China \\
}
\date{}

\maketitle \vspace{0.3cm}
\noindent
\begin{abstract}
 This paper deals with a boundary-value problem for a coupled  quasilinear chemotaxis--haptotaxis model with nonlinear diffusion
$$
 \left\{\begin{array}{ll}
  u_t=\nabla\cdot(D(u)\nabla u)-\chi\nabla\cdot(u\nabla v)-\xi
  \nabla\cdot(u\nabla w)+\mu u(1-u-w),\\
 \disp{v_t=\Delta v- v +u},\quad
\\
\disp{w_t=- vw},\quad\\
 \end{array}\right.
$$
in $N$-dimensional smoothly bounded domains, where the parameters $\xi ,\chi> 0$,
$\mu> 0$. The diffusivity $D(u)$ is assumed to satisfy  $D(u)\geq C_{D}u^{m-1}$ for all $u > 0$ with
some  $C_D>0$.
Relying on a new energy inequality, in this paper,
it is proved that under the conditions
$$m>\frac{2N}{N+{{{\frac{(\frac{\max_{s\geq1}\lambda_0^{\frac{1}{{{s}}+1}}
(\chi+\xi\|w_0\|_{L^\infty(\Omega)})}{(\max_{s\geq1}\lambda_0^{\frac{1}{{{s}}+1}}(\chi+\xi\|w_0\|_{L^\infty(\Omega)})-\mu)_{+}}+1)
(N+\frac{\max_{s\geq1}\lambda_0^{\frac{1}{{{s}}+1}}(\chi+\xi\|w_0\|_{L^\infty(\Omega)})}{(\max_{s\geq1}\lambda_0^{\frac{1}{{{s}}+1}}
(\chi+\xi\|w_0\|_{L^\infty(\Omega)})-\mu)_{+}}-1)}{N}}}}},$$
and proper regularity hypotheses on the initial data, the corresponding initial-boundary problem possesses
at least one global bounded classical solution when $D(0) > 0$ (the case of non-degenerate diffusion), while if, $D(0)\geq 0$ (the case of possibly degenerate
diffusion), the existence of bounded weak solutions for system is shown. This extends some recent
results  by several authors.
\end{abstract}

\vspace{0.3cm}


\newpage
\section{Introduction}
Cancer invasion is a very complex process which involves various biological mechanisms (see \cite{Bellomo,Horstmann2710,Chaplain3,Chaplain1,Friedman,Hillen79}).
Chemotaxis is the oriented movement of cells along concentration gradients of chemicals produced by the cells themselves or in their environment, and
is a significant mechanism of directional migration of cells.  A well-known chemotaxis model was proposed by
Keller and Segel (\cite{Keller79,Keller791}) in the 1970s, which describes the aggregation processes of the cellular
slime mold Dictyostelium discoideum.
 Since then, the following quasi-chemotaxis-only model
 \begin{equation}
 \left\{\begin{array}{ll}
  u_t=\Delta u-\chi\nabla\cdot(u\nabla v)+\mu u(1-u),\quad
x\in \Omega, t>0,\\
 \disp{  v_t=\Delta v +u- v},\quad
x\in \Omega, t>0\\
 \end{array}\right.\label{7101fhhssdsddd.2x1}
\end{equation}
  and its variations have  been widely studied by many authors, where the main issue of the investigation was
whether the solutions to the models are bounded or blow-up (see e.g., Herrero and  Vel\'{a}zquez \cite{Herrero710}, Nagai et al. \cite{Nagaixcdf791},
Winkler et al. \cite{Winkler792,Winkler793},  the survey  \cite{Bellomo1216}). For example, 
 as we all   known that all solutions of \dref{7101fhhssdsddd.2x1} are global in time
and bounded when either $N\geq 3$ and $\mu> 0$ is sufficiently large (see \cite{Winkler37103} and also \cite{Zhengssdddssddddkkllssssssssdefr23}), or $N= 2$ and $\mu> 0$
is arbitrary (\cite{Osakix391}).
Tello and  Winkler (\cite{Tello710})   proved that
the global boundedness for parabolic-elliptic  chemotaxis-only system
 \dref{7101fhhssdsddd.2x1} (the second equation of \dref{7101fhhssdsddd.2x1} is replaced by $-\Delta v+v =u$)
 exists
 under the condition $\mu > \frac{(N-2)^+}{N}\chi$, moreover,  they gave the  weak solutions
 for arbitrary small $\mu > 0$. Some recent
studies show that nonlinear chemotactic sensitivity functions (\cite{Calvez710,Horstmann791,Hillen5662710}), nonlinear
diffusion (\cite{Kowalczyk7101,Cie72,Sugiyama710}), or also logistic dampening (\cite{Osakix391,Tello710,Winkler79,Winkler37103}) may
prevent blow-up of solutions.

 One important extension of the classical Keller-Segel model to a more complex cell migration
mechanism was proposed by Chaplain and Lolas (\cite{Chaplain1,Chaplain7}) in order to describe processes of cancer
invasion. In fact, let $u = u(x,t)$ denote
the density of the tumour cell population, $v = v(x,t)$ represent the concentration
of a matrix-degrading enzyme (MDE) and  $w = w(x,t)$ stand for the density
of the surrounding tissue  (extracellular matrix (ECM)). Then
Chaplain and Lolas (\cite{Chaplain7}) introduced the following chemotaxis-haptotaxis system as a model
describing the process of cancer invasion
\begin{equation}
 \left\{\begin{array}{ll}
  u_t=D\Delta u-\chi\nabla\cdot(u\nabla v)-\xi\nabla\cdot
  (u\nabla w)+\mu u(1-u-w),\quad
x\in \Omega, t>0,\\
 \disp{\tau v_t=\Delta v +u- v},\quad
x\in \Omega, t>0,\\
\disp{w_t=- vw+\eta w(1-u-w) },\quad
x\in \Omega, t>0,\\
 \end{array}\right.\label{7101.2x19xjjssdkkkkk3189}
\end{equation}
where %
 $\tau\in\{0,1\},$
 $\Delta=\disp{\sum_{i=1}^N\frac{\partial^2}{\partial x^2_i}}$, $\disp\frac{\partial}{\partial\nu}$ denotes the outward normal
derivative on $\partial\Omega$,  $\chi>0$ and $\xi>0$ measure the chemotactic and haptotactic
sensitivities, respectively. Here $D>0$ as well as $\mu>0$  and $\eta\geq0$ represent the random motility coefficient, the proliferation rate of the cells and the  remodeling
rate,  respectively. Model \dref{7101.2x19xjjssdkkkkk3189} and its analogue have been extensively studied up to now (see \cite{Tao2,Taox26,Taossddssd793,Tao79477,Taox26216,Tao3,Tao72,LiLiLi791,Cao,Bellomo1216,Szymaska,Zhengssdddsssddfghhhddddkkllssssssssdefr23}.
In fact, Global existence and asymptotic behavior of solutions to the haptotaxis-only system ($\chi = 0$ in the first equation of \dref{7101.2x19xjjssdkkkkk3189}) have been investigated in \cite{Stinnerddff12,Walker,Lianu1,Marciniak} and \cite{Taossddssd793} for the case $\eta= 0$ and $\eta\neq 0$, respectively.

In realistic situations,  the renewal of
the ECM occurs at much smaller timescales than its degradation (see \cite{Tao2,Perthame317,Jger317,Lianu1,Marciniak,Tao79477,Walker}). Therefore, a choice of  $\eta= 0$ on \dref{7101.2x19xjjssdkkkkk3189} seems justified (see \cite{Tao2,Perthame317,Jger317,Lianu1,Marciniak,Tao79477,Walker}). The models mentioned above described the random part of the motion of cancer cells by
linear diffusion, however, from a physical point of view migration of the cancer cells through the ECM should rather be regarded like movement
in a porous medium, and so we are led to considering the cell motility $D$ a nonlinear
function of the cancer cell density.   Inspired by the analysis, in this paper,we consider the following  chemotaxis-haptotaxis system with nonlinear diffusion
(see also \cite{Tao72,Bellomo1216,Wangscd331629})
\begin{equation}
 \left\{\begin{array}{ll}
  u_t=\nabla\cdot(D(u)\nabla u)-\chi\nabla\cdot(u\nabla v)-\xi\nabla\cdot
  (u\nabla w)+\mu u(1-u-w),\quad
x\in \Omega, t>0,\\
 \disp{v_t=\Delta v +u- v},\quad
x\in \Omega, t>0,\\
\disp{w_t=- vw },\quad
x\in \Omega, t>0,\\
 \disp{\frac{\partial u}{\partial \nu}=\frac{\partial v}{\partial \nu}=\frac{\partial w}{\partial \nu}=0},\quad
x\in \partial\Omega, t>0,\\
\disp{u(x,0)=u_0(x)},v(x,0)=v_0(x),w(x,0)=w_0(x),\quad
x\in \Omega\\
 \end{array}\right.\label{7101.2x19x3sss189}
\end{equation}
in a bounded domain $\Omega\subset R^N (N\geq1)$ with smooth boundary $\partial\Omega$. The origin of the system was proposed by
Chaplain and Lolas (\cite{Chaplain1,Chaplain7}) to describe cancer cell invasion into surrounding healthy tissue. Here we assume that $D( u )$ is a nonlinear function
and satisfies
\begin{equation}\label{9161}
D\in  C^{2}([0,\infty)) ~~\mbox{and}~~D(u)\geq C_{D}u^{m-1}~~ \mbox{for all}~ u>0
\end{equation}
with some  $C_{D}>0$ and $m > 0$.
Moreover, if,
 $D ( u )$  fulfills
\begin{equation}\label{91ssdd61}
D ( u ) > 0 ~~ \mbox{for all}~ u\geq0,
\end{equation}
so the diffusion is nondegenerate and the solutions may be considered in the sense of classical.
Throughout this paper, the initial data $(u_0,v_0,w_0)$ are assumed
that for some $\vartheta\in(0,1)$
 \begin{equation}\label{x1.731426677gg}
\left\{
\begin{array}{ll}
\displaystyle{u_0\in C(\bar{\Omega})~~\mbox{with}~~u_0\geq0~~\mbox{in}~~\Omega~~\mbox{and}~~u_0\not\equiv0},\\
\displaystyle{v_0\in W^{1,\infty}(\Omega)~~\mbox{with}~~v_0\geq0~~\mbox{in}~~\Omega},\\
\displaystyle{w_0\in C^{2+\vartheta}(\bar{\Omega})~~\mbox{with}~~w_0>0~~\mbox{in}~~\bar{\Omega}~~\mbox{and}~~\frac{\partial w_0}{\partial\nu}=0~~\mbox{on}~~\partial\Omega.} ~~~~~~~~~~~~~~~~~~~~~~~~~~~~~~~~~~~~~~~~~~~~~~~~~~~~~~~~~~~~~~~~~~~\\
\end{array}
\right.
\end{equation}
System \dref{7101.2x19x3sss189} has been widely studied by many authors, where the main issue of the investigation was
whether the solutions to the models are bounded or blow-up (see Tao-Winkler \cite{Tao72}, ).
For instance, when $D$
satisfies \dref{9161}--\dref{91ssdd61}, Tao and Winkler (\cite{Tao72}) showed that model \dref{7101.2x19x3sss189}  has global solutions provided that
$m>\max\{1,\bar{m}\}$,
where
\begin{equation}\label{dcfvgg7101.2x19x318jkl}
\bar{m}:=\left\{\begin{array}{ll}
\frac{2N^2+4N-4}{N(N+4)}~~\mbox{if}~~
N \leq 8,\\
 \frac{2N^2+3N+2-\sqrt{8N(N+1)}}{N(N+20)}~~\mbox{if}~~
N \geq 9.\\
 \end{array}\right.
\end{equation}
However, they leave a question here: ``whether the
global solutions are bounded''. If $N\geq 2$, the global boundedness of solutions to \dref{7101.2x19x3sss189} has been constructed for $m > 2-\frac{2}{N}$ (see \cite{LiLiLi791,Wangscd331629}) with the help of the boundness of $\|\nabla v\|_{L^{l}(\Omega\times(0,T))}(1\leq l<\frac{N}{N-1})$).
Recently, we  (\cite{Zhddengssdeeezseeddd0}) extended these results to the cases
$m >\frac{2N}{N+2}$ by using the boundness of $\|\nabla v\|_{L^{2}(\Omega\times(0,T))}$. More recently, if $\frac{\mu}{\chi}$ is large enough, Jin \cite{Jineeezseeddd0} (see also \cite{HuHuHueeezseeddd0}) proved that system \dref{7101.2x19x3sss189} admits a
 global bounded  solution for  any $m>0$. However,   we should point that the cases
$0<m\leq\frac{2N}{N+2}$ and small $\frac{\mu}{\chi}$
remain unknown even in the case for  the chemotaxis-only system \dref{7101.2x19x3sss189},
 that is, $w\equiv0$ in system \dref{7101.2x19x3sss189}. In this paper, we firstly use the boundedness  of $\int_{(t-1)_+}^t
 \int_\Omega u^{{\gamma_0+1}}$ (see Lemma \ref{qqqqlemma45630}) for some $\gamma_0>1,$ which is a new result even for chemotaxis-only system \dref{7101.2x19x3sss189}. Then, applying the  standard testing procedures,
we can  derive the uniform boundedness of  $\nabla v$ in $L^{l_0} (\Omega)$ for some $l_0>2.$
 We emphasize that the spontaneous boundedness information on $\nabla v$ in $L^{l_0} (\Omega)$  (see \dref{zjscz2.5297x9ssd630xxy}) plays a key role in this process. Using  the $L^{l_0}$-boundedness of $\nabla v$  and $L^{1}$-boundedness of $u$, we can then acquire the uniform bounds of $u$
in
arbitrary large  $L^p(\Omega)$ provided that the further restriction on $m$ is satisfied (see the proof of Lemmas \ref{lsssemma456302ssd23116}-\ref{sddlemmddffa45630}). Finally,  combining with Moser iteration
method and $L^p$-$L^q$ estimates for Neumann heat semigroup, we finally established the $L^\infty$ bound of $u$ (see Lemmas \ref{lemdffmddffa45630}--\ref{lemdffssddmddffa45630}).


 Motivated by the above works, this paper will focus on studying the relationship between
the exponent $m$ and the global existence of solutions to  chemotaxis-haptotaxis model \dref{7101.2x19x3sss189}  with nonlinear diffusion.
In fact,
  the aim of the present paper is to
   study the quasilinear chemotaxis system \dref{7101.2x19x3sss189} under the conditions \dref{9161}--\dref{91ssdd61}.
    For non-degenerate and
degenerate diffusion both, we will show the existence of global-in-time solutions to system \dref{7101.2x19x3sss189}
that are uniformly bounded. The main results are as follows.
\begin{theorem}\label{theorem3}
Let $\Omega\subset R^N (N\geq1)$ be a bounded domain with smooth boundary
and  $\chi>0,\xi>0,\mu>0$.
Assume that the nonnegative initial data $(u_0 ,v_0,w_0)$
fulfill \dref{x1.731426677gg}. Moreover, if $D$ satisfies \dref{9161}-\dref{91ssdd61} with
 \begin{equation}\label{7101.2ssddffx19x3sss189}
 m>\frac{2N}{N+{{{\frac{(\frac{\max_{s\geq1}\lambda_0^{\frac{1}{{{s}}+1}}(\chi+\xi\|w_0\|_{L^\infty(\Omega)})}{(\max_{s\geq1}
 \lambda_0^{\frac{1}{{{s}}+1}}(\chi+\xi\|w_0\|_{L^\infty(\Omega)})-\mu)_{+}}+1)(N+\frac{\max_{s\geq1}\lambda_0^{\frac{1}{{{s}}+1}}
 (\chi+\xi\|w_0\|_{L^\infty(\Omega)})}{(\max_{s\geq1}\lambda_0^{\frac{1}{{{s}}+1}}(\chi+\xi\|w_0\|_{L^\infty(\Omega)})-\mu)_{+}}-1)}{N}}}}},
 \end{equation}
%
%
%
%
then there exists a triple $(u,v,w)\in (C^0(\bar{\Omega}\times[0,\infty))\cap C^{2,1}
(\bar{\Omega}\times(0,\infty)))^3$ which solves \dref{7101.2x19x3sss189} in the classical sense.
Moreover, both $u$, $v$  and $w$ are bounded in $\Omega\times(0,\infty)$, that is, there exists a positive constant $C$ such that
\begin{equation}\label{916sdfff1}
\|u(\cdot, t)\|_{L^\infty(\Omega)}+\|v(\cdot, t)\|_{W^{1,\infty}(\Omega)}+\|w(\cdot, t)\|_{L^\infty(\Omega)}\leq C~~~\mbox{for all}~~t>0.
\end{equation}
\end{theorem}
\begin{remark}
(i) 
Obviously,
 $$\mu_*=\frac{\max_{s\geq1}\lambda_0^{\frac{1}{{{s}}+1}}(\chi+\xi\|w_0\|_{L^\infty(\Omega)})}{[\max_{s\geq1}\lambda_0^{\frac{1}{{{s}}+1}}
 (\chi+\xi\|w_0\|_{L^\infty(\Omega)})-\mu]_{+}}>1~~(\mbox{by using}~~~ \mu>0),$$
then
 $\gamma_{**}={{{\frac{(\mu_{*}+1)(N+\mu_{*}-1)}{N}}}}=\mu_{*}+1+\frac{\mu_{*}^2-1}{N}>\mu_{*}+1>2$, hence $$\frac{2N}{N+{{{\frac{(\frac{\chi\max\{1,\lambda_0\}}{(\chi\max\{1,\lambda_0\}-\mu)_{+}}+1)(N+\frac{\chi\max\{1,\lambda_0\}}{(\chi\max\{1,\lambda_0\}-\mu)_{+}}-1)}{N}}}}}
<\frac{2N}{N+2}\leq2-\frac{2}{N},$$
therefore,
Theorem \ref{theorem3}
extends the results of  Theorem 1.1  of Zheng (\cite{Zhddengssdeeezseeddd0}), the results of Theorem 1.1   of Wang (\cite{Wangscd331629}), the results of  as well as of Li-Lankeit (\cite{LiLiLi791})
 and partly extends the results of  Theorem 1.1  of Liu et al (\cite{Liughjj791}). Here  the
 assumption $m>\frac{2N}{N+2}$ (see \cite{Zhddengssdeeezseeddd0}) or $m>2-\frac{2}{N}$ (see \cite{LiLiLi791,Liughjj791,Wangscd331629})
are intrinsically required.


(ii) Obviously, for any $N\geq1$,
$$
\frac{2N}{N+{{{\frac{(\frac{\chi\max\{1,\lambda_0\}}{(\chi\max\{1,\lambda_0\}-\mu)_{+}}+1)(N+\frac{\chi\max\{1,\lambda_0\}}
{(\chi\max\{1,\lambda_0\}-\mu)_{+}}-1)}{N}}}}}
<\bar{m},$$ therefore, Theorem \ref{theorem3} extends the results of  Corollary 1.2  of   Tao and Winkler (\cite{Tao72}), who showed the
global existence
of solutions
 the cases  $m>\bar{m},$ where $\bar{m}$ is given by \dref{dcfvgg7101.2x19x318jkl}.


 (iii) In the case $N= 2$, by using $\mu>0,$ then $\frac{8}{4+(\mu_*+1)^2}< 1$, our result improves the result of \cite{Taox3201} and \cite{zhengjjkk}, in which the
 assumption $m=1$ or $m>1$
are intrinsically required.


 (iv)
If $\mu>\max_{s\geq1}
 \lambda_0^{\frac{1}{{{s}}+1}}(\chi+\xi\|w_0\|_{L^\infty(\Omega)})$, then by \dref{7101.2ssddffx19x3sss189}, we derive that for any $m>0,$ system \dref{7101.2x19x3sss189} has
a classical and bounded  solution, which improves the result of  \cite{Jineeezseeddd0} as well as  \cite{HuHuHueeezseeddd0} and \cite{Cao}.

 (v)  The chemotaxis-haptotaxis system therefore has bounded solutions under the
same condition on $m$ as the pure chemotaxis system with  $w\equiv 0$ without logistic source (see \cite{Tao794}).
For  $\mu= 0$ this condition is essentially optimal (\cite{Winkler79}).

\end{remark}

In the case of possibly degenerate diffusion, system
\dref{7101.2x19x3sss189} admits  at least one global bounded weak solution:

\begin{theorem}\label{theossdrem3}
Let $\Omega\subset R^N (N\geq1)$ be a bounded domain with smooth boundary
and  $\chi>0,\xi>0,\mu>0$.
Suppose that the initial data   $(u_0 ,v_0,w_0)$
satisfy  \dref{x1.731426677gg}. Moreover, if $D$ satisfies \dref{9161} with
  $$m>\frac{2N}{N+{{{\frac{(\frac{\max_{s\geq1}\lambda_0^{\frac{1}{{{s}}+1}}(\chi+\xi\|w_0\|_{L^\infty(\Omega)})}{(\max_{s\geq1}\lambda_0^{\frac{1}{{{s}}+1}}(\chi+\xi\|w_0\|_{L^\infty(\Omega)})-\mu)_{+}}+1)(N+\frac{\max_{s\geq1}\lambda_0^{\frac{1}{{{s}}+1}}(\chi+\xi\|w_0\|_{L^\infty(\Omega)})}{(\max_{s\geq1}\lambda_0^{\frac{1}{{{s}}+1}}(\chi+\xi\|w_0\|_{L^\infty(\Omega)})-\mu)_{+}}-1)}{N}}}}},$$
%
%
%
%
then system \dref{7101.2x19x3sss189} admits at
least one global weak solution $(u,v,w)$ in the sense of definition \ref{df1} below that exists globally in time and is bounded in the sense that
\dref{916sdfff1} holds.
%
%
\end{theorem}

%

The rest of this paper is organized as follows. In the following section,  we  recall some preliminary results.  Section 3 is devoted to a series of a priori estimates and then prove Theorem \ref{theorem3}.
In Section 4, applying   the existence of classical solutions in the non-degenerate case,  we will
then complete the proof of theorem \ref{theossdrem3} 
 by
an approximation procedure in Section 3.

\section{Preliminaries and  main results}

Before proving our main results, we will give some preliminary
lemmas, which play a crucial role in the following proofs. As for
the proofs of these lemmas, here we will not repeat them again.

\begin{lemma}(\cite{Hajaiej,Ishida})\label{lemma41ffgg}
Let  $s\geq1$ and $q\geq1$.
Assume that $p >0$ and $a\in(0,1)$ satisfy
$$\frac{1}{2}-\frac{p}{N}=(1-a)\frac{q}{s}+a(\frac{1}{2}-\frac{1}{N})~~\mbox{and}~~p\leq a.$$
Then there exist $c_0, c'_0 >0$ such that for all $u\in W^{1,2}(\Omega)\cap L^{\frac{s}{q}}(\Omega)$,
$$\|u\|_{W^{p,2}(\Omega)} \leq c_{0}\|\nabla u\|_{L^{2}(\Omega)}^{a}\|u\|^{1-a}_{L^{\frac{s}{q}}(\Omega)}+c'_0\|u\|_{L^{\frac{s}{q}}(\Omega)}.$$
%
\end{lemma}
\begin{lemma}(\cite{Zheng})\label{lemma41}
Let $0<{\theta}\leq p\leq\frac{2N}{N-2}$.
There exists a positive constant $C_{GN}$ such that for all $u \in W^{1,2}(\Omega)\cap L^{{\theta}}(\Omega)$,
$$\|u\|_{L^p(\Omega)} \leq C_{GN}(\|\nabla u\|_{L^{2}(\Omega)}^{a}\|u\|^{1-a}_{L^{{\theta}}(\Omega)}+\|u\|_{L^{{\theta}}(\Omega)})$$
is valid with 
$a =\disp{\frac{\frac{N}{{{\theta}}}-\frac{N}{p}}{1-\frac{N}{2}+\frac{N}{{\theta}}}}\in(0,1)$.
%
\end{lemma}

\begin{lemma}\label{lemma45xy1222232} (\cite{Hieber,Zhenddddgssddsddfff00})
Suppose  that $\gamma\in (1,+\infty)$ and $g\in L^\gamma((0, T); L^\gamma(
\Omega))$.
 Consider the following evolution equation
 $$
 \left\{\begin{array}{ll}
v_t -\Delta v+v=g,~~~(x, t)\in
 \Omega\times(0, T ),\\
\disp\frac{\partial v}{\partial \nu}=0,~~~(x, t)\in
 \partial\Omega\times(0, T ),\\
v(x,0)=v_0(x),~~~(x, t)\in
 \Omega.\\
 \end{array}\right.
 $$
 For each $v_0\in W^{2,\gamma}(\Omega)$
such that $\disp\frac{\partial v_0}{\partial \nu}=0$, there exists a unique solution
$v\in W^{1,\gamma}((0,T);L^\gamma(\Omega))\cap L^{\gamma}((0,T);W^{2,\gamma}(\Omega)).$ In addition, if $s_0\in[0,T)$, $v(\cdot,s_0)\in W^{2,\gamma}(\Omega)(\gamma>N)$ with $\disp\frac{\partial v(\cdot,s_0)}{\partial \nu}=0,$
then there exists a positive constant $\lambda_0:=\lambda_0(\Omega,\gamma,N)$ such that  
$$
\begin{array}{rl}
&\disp{\int_{s_0}^Te^{\gamma s}\| v(\cdot,t)\|^{\gamma}_{W^{2,\gamma}(\Omega)}ds\leq\lambda_0\left(\int_{s_0}^Te^{\gamma s}
\|g(\cdot,s)\|^{\gamma}_{L^{\gamma}(\Omega)}ds+e^{\gamma s_0}(\|v_0(\cdot,s_0)\|^{\gamma}_{W^{2,\gamma}(\Omega)})\right).}\\
\end{array}
$$
\end{lemma}

The following local existence result is rather standard; since a similar reasoning in
\cite{Tao72,Zheng},
see for example.  Therefore, we only give the
following lemma without proof.

\begin{lemma}\label{lemma70}
Assume that the nonnegative functions $u_0,v_0,$ and $w_0$ satisfies \dref{x1.731426677gg}
for some $\vartheta\in(0,1),$
$D$ satisfies \dref{9161} and \dref{91ssdd61}.
%
%
 Then there exists a maximal existence time $T_{max}\in(0,\infty]$ and a triple of  nonnegative functions $$
 (u ,v ,w )\in C^0(\bar{\Omega}\times[0,T_{max}))\cap C^{2,1}(\bar{\Omega}\times(0,T_{max}))\times
 C^0((0,T_{max}); C^2(\bar{\Omega}))\times C^{2,1}(\bar{\Omega}\times[0,T_{max}))$$
 which solves \dref{7101.2x19x3sss189}  classically and satisfies $0\leq w \leq \|w_0\|_{L^\infty(\Omega)}$
  in $\Omega\times(0,T_{max})$.
%
Moreover, if  $T_{max}<+\infty$, then
\begin{equation}
\left(|u (\cdot, t)\|_{L^\infty(\Omega)}+\|v (\cdot, t)\|_{W^{1,\infty}(\Omega)}+\|w (\cdot, t)\|_{W^{1,\infty}(\Omega)}\right)\rightarrow\infty~~ \mbox{as}~~ t\nearrow T_{max}.
\label{1.163072x}
\end{equation}
\end{lemma}

 According to the above existence theory,  for any $s\in(0, T_{max})$, $(u(\cdot, s), v(\cdot, s),w(\cdot, s))\in C^2(\bar{\Omega})$.
Without
 loss of generality, we can assume that there exists a positive constant 
$K$ such that
\begin{equation}\label{eqx45xx12112}
\|u_0\|_{C^2(\bar{\Omega})}\leq K~\mbox{as well as }~~~\|v_0\|_{C^2(\bar{\Omega})}\leq K~~\mbox{and}~~\|w_0\|_{C^2(\bar{\Omega})}\leq K.
\end{equation}

\section{A priori estimates}
The main task of this section is to establish for estimates for the solutions
$(u ,v , w )$ of  problem \dref{7101.2x19x3sss189}.
To this end, in straightforward fashion one can check the following boundedness for $u$, which is common in
 chemotaxis (or chemotaxis--haptotaxis)  with logistic source  (see e.g. \cite{Zhddengssdeeezseeddd0,Winkler37103,Wangscd331629,LiLiLi791}).
\begin{lemma}\label{wsdelemma45}
There exists 
$C > 0$ such that the solution of \dref{7101.2x19x3sss189} satisfies
%
%
\begin{equation}
\int_{\Omega}{u }+\int_{\Omega}|\nabla v |^2+\int_{\Omega}|\nabla v |^l \leq C~~\mbox{for all}~~ t\in(0, T_{max})
\label{cz2.5ghju48cfg924ghyuji}
\end{equation}
with $l\in[1,\frac{N}{N-1}).$


Since,  the third component of \dref{7101.2x19x3sss189} can be expressed explicitly in terms of
$v $. This leads to the following 
%
 a one-sided pointwise estimate for $-\Delta w $ (see e.g. \cite{Tao79477ddffvg,Taox3201,Tao72}):
%
\begin{lemma}\label{lemm3a} Let $(u , v ,w )$ solve \dref{7101.2x19x3sss189} in $\Omega\times(0, T_{max})$.
 Then
 \begin{equation}\label{x1.731426677gghh}
 \begin{array}{rl}
-\Delta w (x, t) \leq&\disp{ \|w_0\|_{L^\infty(\Omega)}\cdot v (x,t)+\kappa~~~\mbox{for all}~~x\in\Omega~~\mbox{and}~~~t\in(0, T_{max}),}\\
\end{array}
\end{equation}
where
\begin{equation}\label{x1.73ddff1426677gghh}
\kappa:=\|\Delta w_0\|_{L^\infty(\Omega)}+4\|\nabla\sqrt{w_0}\|_{L^\infty(\Omega)}^2+\frac{\|w_0\|_{L^\infty(\Omega)}}{e}.
\end{equation}
\end{lemma}
%
%
\end{lemma}
Now we proceed to establish the main step towards our boundedness proof. 
To this end, let us collect some basic
estimates for $u$ and $v$ in comparatively large function spaces. In fact,
relying on a standard testing procedure, we derive the following Lemma: 
%
\begin{lemma}\label{zsxxcdvlemma45630}
For any  $k>1$, the solution $(u ,v ,w )$ of \dref{7101.2x19x3sss189} satisfies that
\begin{equation}
\begin{array}{rl}
&\disp{-\xi\int_\Omega  u ^{k-1}\nabla\cdot(u \nabla w  )\leq \frac{({k-1})}{k}\xi\|w_0\|_{L^\infty(\Omega)}\int_\Omega u^k  v  dx+\kappa \frac{({k-1})}{k}\xi \int_\Omega u^k  dx,}\\
\end{array}
\label{vbgncz2.5xx1ffgghh512}
\end{equation}
where $\kappa$ is the same as \dref{x1.73ddff1426677gghh}.
\end{lemma}
\begin{proof}
For any $k>1$,  we integrate the left hand of \dref{vbgncz2.5xx1ffgghh512} and use Lemma \ref{lemm3a} then get
\begin{equation}
\begin{array}{rl}
&\disp{-\xi\int_\Omega   u ^{k-1}\nabla\cdot(u \nabla w )
  dx}
\\
=&\disp{({k-1} )\xi\int_\Omega    u ^{k-1}\nabla u \cdot\nabla w  dx}
\\
=&\disp{-\frac{({k-1})}{k}\xi\int_\Omega u^k \Delta w  dx}
\\
\leq&\disp{\frac{({k-1})}{k}\xi\int_\Omega u^k (\|w_0\|_{L^\infty(\Omega)} v +\kappa) dx,}
\\
\end{array}
\label{1cz2.563019rrtttt12}
\end{equation}
where $\kappa$ is the same as \dref{x1.73ddff1426677gghh}. This directly entails \dref{vbgncz2.5xx1ffgghh512}.
\end{proof}

Due to the presence of logistic source, some useful estimates for $u $
can be derived.

\begin{lemma}\label{ssdeedrfe116lemma70hhjj} (see \cite{Wangscd331629,LiLiLi791,Zhddengssdeeezseeddd0})
Assume that $\mu>0.$
There exists a positive constant  
$ K_0$ such that
 the solution $(u , v ,w )$ of  \dref{7101.2x19x3sss189} satisfies
%
%
\begin{equation}
 \begin{array}{rl}
 \disp\int_{\Omega}u (x,t)dx\leq K_0~~~\mbox{for all}~~t\in (0, T_{max})
\end{array}\label{ssddaqwswddaassffssff3.ddfvbb10deerfgghhjuuloollgghhhyhh}
\end{equation}
and
\begin{equation}
\int_t^{t+\tau}\int_{\Omega}{u ^{2}}\leq  K_0~~\mbox{for all}~~ t\in(0, T_{max}-\tau),
\label{bnmbncz2.5ghhjuyuivvddfggghhbssdddeennihjj}
\end{equation}
where we have set
\begin{equation}
\tau:=\min\left\{1,\frac{1}{6}T_{max}\right\}.
\label{cz2.5ghju48cfg924vbhu}
\end{equation}

%
\end{lemma}
In order to establish some estimates for solution $(u ,v , w )$, we first
recall the following  lemma proved in \cite{Wangscd331629} (see also \cite{LiLiLi791,Zhddengssdeeezseeddd0}).
\begin{lemma}\label{lemmssdda45630}
Let $\Omega\subset \mathbb{R}^N(N\geq1)$ be a bounded domain with smooth boundary. 
 Then for all $k>1,$ the solution $(u ,v ,w )$ of  \dref{7101.2x19x3sss189} satisfies that
\begin{equation}\label{cz2.5xx1jjjj}
\begin{array}{rl}
&\disp{\frac{1}{k}\frac{d}{dt}\| u  \|^{k}_{L^k(\Omega)}+\frac{(k-1) C_D }{2}\int_{\Omega} u  ^{m+k-3}|\nabla u  |^2{}+
\mu\int_{\Omega} u  ^{k+1}{}}
\\
\leq&\disp{\frac{\chi^2(k-1)} {2 C_D }\int_{\Omega} u  ^{k+1-m}|\nabla v  |^2+
\frac{({k-1})}{k}\xi\|w_0\|_{L^\infty(\Omega)}\int_\Omega  u ^{k}v +(\mu+\kappa\xi)\int_{\Omega} u ^{k}}\\
\end{array}
\end{equation}
for all $t\in(0, T_{max})$.
\end{lemma}
\begin{proof}
Multiplying
 $\dref{7101.2x19x3sss189}_1$ (the first equation of \dref{7101.2x19x3sss189})
  by $ u ^{k-1}$ and integrating over $\Omega$, 
 we get
\begin{equation}
\begin{array}{rl}
&\disp{\frac{1}{k}\frac{d}{dt}\| u  \|^{k}_{L^k(\Omega)}+  C_D(k-1)\int_{\Omega} u ^{m+k-3}|\nabla u  |^2{}}
\\
\leq&\disp{-\chi\int_\Omega \nabla\cdot( u \nabla v  )  u ^{k-1}{}-\xi\int_\Omega\nabla\cdot
  (  u  \nabla w )   u ^{k-1}{}}\\
  &+\disp{\mu\int_\Omega    u ^{k}(1-u -w  ) {}}\\
  \leq&\disp{-\chi\int_\Omega \nabla\cdot( u \nabla v  )  u ^{k-1}{}-\xi\int_\Omega\nabla\cdot
  (  u  \nabla w )   u ^{k-1}{}}\\
  &+\disp{\mu\int_\Omega    u ^{k}(1-u ) {}~~\mbox{for all}~~ t\in(0,T_{max})}\\
\end{array}
\label{cfvgvvcz2.5xx1jjjj}
\end{equation}
according to the nonnegativity of $w$.
Integrating by parts to the first term on the right hand side of \dref{cfvgvvcz2.5xx1jjjj} and using %
the Young inequality, we obtain
\begin{equation}
\begin{array}{rl}
&\disp{-\chi\int_\Omega \nabla\cdot( u \nabla v  )  u ^{k-1}{}}
\\
=&\disp{(k-1 )\chi\int_\Omega    u ^{k-1}\nabla u  \cdot\nabla v  {}}
\\
\leq&\disp{\frac{(k-1) C_D }{2}\int_{\Omega} u  ^{m+k-3}|\nabla u  |^2{} +\frac{\chi^2(k-1)} {2 C_D }\int_{\Omega} u  ^{k+1-m}|\nabla v  |^2.}\\
\end{array}
\label{111cz2.5630111}
\end{equation}
On the other hand, due to Lemma \ref{zsxxcdvlemma45630}, we have
\begin{equation}
\begin{array}{rl}
&\disp{-\xi\int_\Omega \nabla\cdot(  u  \nabla w )
   u ^{k-1} {}}
\\
\leq&\disp{\frac{({k-1})}{k}\xi\|w_0\|_{L^\infty(\Omega)}\int_\Omega u^k  v  dx+\kappa \frac{({k-1})}{k}\xi \int_\Omega u^k  dx{}}\\
\leq&\disp{\frac{({k-1})}{k}\xi\|w_0\|_{L^\infty(\Omega)}\int_\Omega u^k  v  dx+\kappa\xi \int_\Omega u^k  dx{}.}\\
\end{array}
\label{cz2.563019rrtttt12}
\end{equation}
Furthermore, inserting \dref{111cz2.5630111}--\dref{cz2.563019rrtttt12} into \dref{cfvgvvcz2.5xx1jjjj}, we conclude that for all $t\in(0, T_{max})$,
\begin{equation}
\begin{array}{rl}
&\disp{\frac{1}{k}\frac{d}{dt}\| u  \|^{k}_{L^k(\Omega)}+\frac{(k-1) C_D }{2}\int_{\Omega} u  ^{m+k-3}|\nabla u  |^2{}+
\mu\int_{\Omega} u  ^{k+1}{}}
\\
\leq&\disp{\frac{\chi^2(k-1)} {2 C_D }\int_{\Omega} u  ^{k+1-m}|\nabla v  |^2+
\frac{({k-1})}{k}\xi\|w_0\|_{L^\infty(\Omega)}\int_\Omega  u ^{k}v +(\mu+\kappa\xi)\int_{\Omega} u ^{k}.}\\
\end{array}
\label{vgbhnsxcdvfcz2.5xx1jjjj}
\end{equation}
\end{proof}
We proceed to estimate both integrals on the right of \dref{cz2.5xx1jjjj} in a straightforward manner.
\begin{lemma}\label{qqqqlemma45630}
Let $(u ,v ,w )$ be a solution to \dref{7101.2x19x3sss189} on $(0,T_{max})$ and
\begin{equation}\mu_{*}=\frac{\max_{s\geq1}\lambda_0^{\frac{1}{{{s}}+1}}(\chi+\xi\|w_0\|_{L^\infty(\Omega)})}{\left[\max_{s\geq1}
 \lambda_0^{\frac{1}{{{s}}+1}}(\chi+\xi\|w_0\|_{L^\infty(\Omega)})-\mu\right]_{+}}.
 \label{zjscz2.ddffrr5297x96ssddffdffggbh302222114}
\end{equation}
 If  $\mu>0,$
  then
 for all $1<\gamma_0<\mu_{*}$,
there exists a positive constant $C$ which depends on $\gamma_0$ such that 
\begin{equation}
\int_{\Omega}u ^{\gamma_0}(x,t) \leq C ~~~\mbox{for all}~~ t\in(0,T_{max})
\label{zjscz2.ddffrr5297x96302222114}
\end{equation}
and
\begin{equation}
\int_{0}^{t}\int_{\Omega}u ^{\gamma_0+1}(x,t) \leq C ~~~\mbox{for all}~~ t\in(0,T_{max}).
\label{4455zjscz2.ddffrr5297x96302222114}
\end{equation}
\end{lemma}
\begin{proof}
Multiplying $\dref{7101.2x19x3sss189}_1$
  by $u ^{k-1}$, integrating over $\Omega$  and using $w\geq0$, 
 we get
\begin{equation}
\begin{array}{rl}
&\disp{\frac{1}{k}\frac{d}{dt}\|u \|^{k}_{L^k(\Omega)}+ C_D(k-1)\int_{\Omega}u ^{m+k-3} |\nabla u |^2dx}
\\
\leq&\disp{-\chi\int_\Omega \nabla\cdot(u \nabla v )u ^{k-1}dx-\xi\int_\Omega\nabla\cdot
  (u \nabla w )u ^{k-1}+\mu
\int_\Omega  u ^{k} (1- u -w ) }\\
\leq&\disp{-\chi\int_\Omega \nabla\cdot(u \nabla v )u ^{k-1}dx-\xi\int_\Omega\nabla\cdot
  (u \nabla w )u ^{k-1}+\mu
\int_\Omega  u ^{k}(1- u ) .}\\
\end{array}
\label{qqqqcz2.5xx1jjjj}
\end{equation}
We now estimate the right hand side of \dref{qqqqcz2.5xx1jjjj} terms by terms.
To this end,
integrating by parts to the first term on the right hand side of \dref{qqqqcz2.5xx1jjjj}, we obtain for any $\varepsilon_1>0,$
\begin{equation}
\begin{array}{rl}
&\disp{-\chi\int_\Omega \nabla\cdot(u \nabla v )u ^{k-1}}
\\
=&\disp{-\frac{({k-1})\chi}{k}\int_\Omega u ^{k}\Delta v  }
\\
\leq&\disp{\frac{(k-1)\chi}{k}\int_\Omega  u ^{k}|\Delta v | }
\\
\leq&\disp{\varepsilon_1\int_\Omega  u ^{k+1}+\gamma_1\varepsilon_1^{-k}\int_\Omega|\Delta v |^{k+1}, }
\\
\end{array}
\label{223444cz2.5630111}
\end{equation}
where  \begin{equation}\gamma_1=\frac{1}{k+1}\left(\frac{k+1}{k}\right)^{-k}\left(\frac{(k-1)\chi}{k}\right)^{k+1}.
\label{22ddf34ddff44cz2.5630111}
\end{equation}
Due to \dref{x1.731426677gghh} and \dref{x1.73ddff1426677gghh}, it follows that for any $\varepsilon_2>0$
\begin{equation}
\begin{array}{rl}
&\disp{-\xi\int_\Omega \nabla\cdot( u  \nabla w )
 u ^{k-1} }
\\
=&\disp{-\frac{({k-1})\xi}{k}\int_\Omega u ^{k}\Delta w  }
\\
\leq&\disp{\kappa\frac{({k-1})\xi}{k}\int_\Omega u ^{k}+\frac{({k-1})\xi\|w_0\|_{L^\infty(\Omega)}}{k}\int_\Omega u ^{k}v }\\
\leq&\disp{\kappa\xi\int_\Omega u ^{k}+\frac{({k-1})\xi\|w_0\|_{L^\infty(\Omega)}}{k}\int_\Omega u ^{k}v }\\
\leq&\disp{\kappa\xi\int_\Omega u ^{k}+\varepsilon_2\int_\Omega u ^{k+1}+\gamma_2\varepsilon_2^{-k}\int_{\Omega}v ^{k+1},}\\
\end{array}
\label{qqqqcz2.563019rrtttt12}
\end{equation}
where
\begin{equation}\gamma_2:=\frac{1}{k+1}\left(\frac{k+1}{k}\right)^{-k}
\left(\frac{({k-1})\xi\|w_0\|_{L^\infty(\Omega)}}{k}\right)^{k+1}
\label{qqqssffffqcz2.563019rrtttt12}
\end{equation}
and $\kappa$ is give by \dref{x1.73ddff1426677gghh}.
On the other hand, in view of $k>1$, we also derive that
\begin{equation}
\begin{array}{rl}
\mu
\disp\int_\Omega  u ^{k}(1- u ) =&\disp{-\mu \int_\Omega   u ^{k+1}+(\mu+\frac{k+1}{k})\int_\Omega  u ^{k}-\frac{k+1}{k}\int_\Omega  u ^{k}}\\
\leq&\disp{-\mu \int_\Omega   u ^{k+1}+(\mu+2)\int_\Omega  u ^{k}-\frac{k+1}{k}\int_\Omega  u ^{k}.}\\
\end{array}
\label{sddddcz2.56301hh}
\end{equation}
Therefore, combined with \dref{223444cz2.5630111},
\dref{qqqqcz2.563019rrtttt12}, \dref{qqqqcz2.5xx1jjjj}
as well as \dref{sddddcz2.56301hh}  and \dref{91ssdd61}, we have
\begin{equation}
\begin{array}{rl}
&\disp{\frac{1}{k}\frac{d}{dt}\|u \|^{k}_{L^k(\Omega)}+ C_D(k-1)\int_{\Omega}u ^{m+k-3}|\nabla u |^2+\frac{k+1}{k}\int_\Omega  u ^{k}}
\\
\leq&\disp{(-\mu +\varepsilon_1+\varepsilon_2 )\int_\Omega u ^{k+1}+\gamma_1\varepsilon_1^{-k}\int_\Omega|\Delta v |^{k+1}+\gamma_2\varepsilon_2^{-k}\int_\Omega v ^{k+1}+C_1\int_\Omega  u ^{k}}\\
\end{array}
\label{cz2.5xx1}
\end{equation}
with $C_1=\kappa\xi+\mu+2 .$
For any $t\in (0,T_{max})$, applying the Gronwall Lemma   to the above inequality shows that
\begin{equation}
\begin{array}{rl}
&\disp{\frac{1}{{k}}\|u (\cdot,t) \|^{{{k}}}_{L^{{k}}(\Omega)}+ C_D(k-1)\int_{0}^t
e^{-( { {k}+1})(t-s)}\int_{\Omega}u ^{m+k-3}|\nabla u |^2}
\\
\leq&\disp{\frac{1}{{k}}e^{-( { {k}+1})t}\|u_0 \|^{{{k}}}_{L^{{k}}(\Omega)}+(\varepsilon_1+\varepsilon_2- \mu)\int_{0}^t
e^{-( { {k}+1})(t-s)}\int_\Omega u ^{{{k}+1}} dxds}\\
&+\disp{\gamma_1\varepsilon_1^{-k}\int_{0}^t
e^{-( { {k}+1})(t-s)}\int_\Omega |\Delta v |^{ {k}+1} dxds+ C_1\int_{0}^t
e^{-( { {k}+1})(t-s)}\int_\Omega u ^{{{k}}} dxds}\\
&\disp{+\gamma_2\varepsilon_2^{-k}\int_{0}^t
e^{-( { {k}+1})(t-s)}\int_\Omega v ^{{{k}+1}} dxds}\\
\leq&\disp{(\varepsilon_1+\varepsilon_2- \mu)\int_{0}^t
e^{-( { {k}+1})(t-s)}\int_\Omega u ^{{{k}+1}} dxds+\gamma_1\varepsilon_1^{-k}\int_{0}^t
e^{-( { {k}+1})(t-s)}\int_\Omega |\Delta v |^{ {k}+1} dxds}\\
&+\disp{\gamma_2\varepsilon_2^{-k}\int_{0}^t
e^{-( { {k}+1})(t-s)}\int_\Omega v ^{{{k}+1}} dxds+
 C_1\int_{0}^t
e^{-( { {k}+1})(t-s)}\int_\Omega u ^{{{k}}} dxds+C_2,}\\
\end{array}
\label{cz2111ddffdfghhhg11.5kk1214114114rrgg}
\end{equation}
where
$$C_2:=C_2({k})=\frac{1}{{k}}\|u_0 \|^{{{k}}}_{L^{{k}}(\Omega)}.$$
Next, a use of Lemma \ref{lemma45xy1222232} and \dref{eqx45xx12112} leads to
\begin{equation}\label{cz2.5kke3456778999ddff9001214114114rrggjjkk}
\begin{array}{rl}
&\disp{\gamma_1\varepsilon_1^{-k}\int_{0}^t
e^{-( { {k}+1})(t-s)}\int_\Omega |\Delta v |^{ {k}+1} dxds}
\\
=&\disp{\gamma_1\varepsilon_1^{-k}e^{-( { {k}+1})t}\int_{0}^t
e^{( { {k}+1})s}\int_\Omega |\Delta v |^{ {k}+1} dxds}\\
\leq&\disp{\gamma_1\varepsilon_1^{-k}e^{-( { {k}+1})t}\lambda_0(\int_{0}^t
\int_\Omega e^{( { {k}+1})s}u ^{ {k}+1} dxds+\|v_0\|^{ {k}+1}_{W^{2, { {k}+1}}(\Omega)})}\\
\end{array}
\end{equation}
and
\begin{equation}\label{cz2.5kk12141141dfggghhh14rrggjjkk}
\begin{array}{rl}
&\disp{\gamma_2\varepsilon_2^{-k}\int_{0}^t
e^{-( { {k}+1})(t-s)}\int_\Omega v ^{ {k}+1} dxds}
\\
=&\disp{\gamma_2\varepsilon_2^{-k}e^{-( { {k}+1})t}\int_{0}^t
e^{( { {k}+1})s}\int_\Omega  v ^{ {k}+1} dxds}\\
\leq&\disp{\gamma_2\varepsilon_2^{-k}e^{-( { {k}+1})t}\lambda_0(\int_{0}^t
\int_\Omega e^{( { {k}+1})s}u ^{ {k}+1} dxds+\|v_0\|^{ {k}+1}_{W^{2, { {k}+1}}(\Omega)})}\\
\end{array}
\end{equation}
for all $t\in(0, T_{max})$, where $\lambda_0$ is the same as Lemma \ref{lemma45xy1222232}.
On the other hand, choosing $\varepsilon_1=\frac{(k-1)\chi}{k+1}\lambda_0^{\frac{1}{k+1}}$ and
 $\varepsilon_2=\frac{(k-1)\xi\|w_0\|_{L^\infty(\Omega)}}{k+1}\lambda_0^{\frac{1}{k+1}}$, with the help of \dref{22ddf34ddff44cz2.5630111} and \dref{qqqssffffqcz2.563019rrtttt12},  a simple calculation  shows that
$$\varepsilon_1+\gamma_1\lambda_0\varepsilon_1^{-k}=\frac{({k}-1)}{{k}}\lambda_0^{\frac{1}{{k}+1}}\chi$$
and
$$\varepsilon_2+\gamma_2\lambda_0\varepsilon_2^{-k}=\frac{({k}-1)}{{k}}\lambda_0^{\frac{1}{{k}+1}}\xi\|w_0\|_{L^\infty(\Omega)},$$
so that,
substituting \dref{cz2.5kke3456778999ddff9001214114114rrggjjkk}--\dref{cz2.5kk12141141dfggghhh14rrggjjkk} into \dref{cz2111ddffdfghhhg11.5kk1214114114rrgg} implies that 
\begin{equation}
\begin{array}{rl}
&\disp{\frac{1}{{k}}\|u(\cdot,t) \|^{{{k}}}_{L^{{k}}(\Omega)}+ C_D(k-1)\int_{0}^t
e^{-( { {k}+1})(t-s)}\int_{\Omega}u ^{m+k-3}|\nabla u |^2}
\\
\leq&\disp{(\varepsilon_1+\gamma_1\lambda_0\varepsilon_1^{-k}+\varepsilon_2+\gamma_2\lambda_0\varepsilon_2^{-k}- \mu)\int_{0}^t
e^{-( {k}+1)(t-s)}\int_\Omega u ^{{{k}+1}} dxds}\\
&+\disp{(\gamma_1\varepsilon_1^{-k}+\gamma_2\varepsilon_2^{-k})e^{-( {k}+1)(t-s_0)}\lambda_0\|v_0\|^{ {k}+1}_{W^{2, { {k}+1}}(\Omega)}+
 C_1\int_{0}^t
e^{-( { {k}+1})(t-s)}\int_\Omega u ^{{{k}}} dxds+C_2}\\
=&\disp{(\frac{({k}-1)}{{k}}\lambda_0^{\frac{1}{{k}+1}}\chi+\frac{({k}-1)}{{k}}\lambda_0^{\frac{1}{{k}+1}}\xi\|w_0\|_{L^\infty(\Omega)}- \mu)\int_{0}^t
e^{-( { {k}+1})(t-s)}\int_\Omega u ^{{{k}+1}} dxds}\\
&+\disp{(\gamma_1\varepsilon_1^{-k}+\gamma_2\varepsilon_2^{-k})e^{-( {k}+1)(t-s_0)}\lambda_0\|v_0\|^{ {k}+1}_{W^{2, { {k}+1}}(\Omega)}+
 C_1\int_{0}^t
e^{-( { {k}+1})(t-s)}\int_\Omega u ^{{{k}}} dxds+C_2}\\
\leq&\disp{[\frac{({k}-1)}{{k}}\max_{s\geq1}\lambda_0^{\frac{1}{{{s}}+1}}(\chi+\xi\|w_0\|_{L^\infty(\Omega)})- \mu]\int_{0}^t
e^{-( {k}+1)(t-s)}\int_\Omega u ^{{{k}+1}} }\\
&+\disp{ C_1\int_{0}^t
e^{-( { {k}+1})(t-s)}\int_\Omega u ^{{{k}}} dxds+C_3}\\
\end{array}
\label{czssddssddffggf2.5kk1214114114rrggkkll}
\end{equation}
with $$C_3=(\gamma_1\varepsilon_1^{-k}+\gamma_2\varepsilon_2^{-k})e^{-( {k}+1)(t-s_0)}\lambda_0\|v_0\|^{ {k}+1}_{W^{2, { {k}+1}}}+C_2.$$

For any $\varepsilon>0,$
we choose $k=\frac{\max_{s\geq1}\lambda_0^{\frac{1}{{{s}}+1}}(\chi+\xi\|w_0\|_{L^\infty(\Omega)})}{(\max_{s\geq1}\lambda_0^{\frac{1}{{{s}}+1}}(\chi+\xi\|w_0\|_{L^\infty(\Omega)})-\mu)_{+}}-\varepsilon.$
Then $$\frac{({k}-1)}{{k}}\max_{s\geq1}\lambda_0^{\frac{1}{{{s}}+1}}(\chi+\xi\|w_0\|_{L^\infty(\Omega)})<\mu.$$
Thus, by using the Young inequality, we derive that
 there exists a positive constant $C_4$
such that
\begin{equation}
\begin{array}{rl}
&\disp{\int_{\Omega}u^{{k}}(x,t) dx\leq C_4~~\mbox{for all}~~t\in (0, T_{max})}\\
\end{array}
\label{cz2.5kk1214114114rrggkklljjuu}
\end{equation}
and
\begin{equation}
\int_{0}^{t}\int_{\Omega}u ^{k+1}(x,t) \leq C_4 ~~~\mbox{for all}~~ t\in(0,T_{max}).
\label{4455zjscz2.ddffrrssss5297x96302222114}
\end{equation}
Thereupon, combining  with  the arbitrariness of $\varepsilon$ and the H\"{o}lder  inequality,  \dref{zjscz2.ddffrr5297x96302222114} and \dref{4455zjscz2.ddffrr5297x96302222114} holds.
The proof of
Lemma \ref{qqqqlemma45630} is completed.
\end{proof}
When  $$\mu\geq\max_{s\geq1}\lambda_0^{\frac{1}{{{s}}+1}}(\chi+\xi\|w_0\|_{L^\infty(\Omega)}),$$
by making use of above lemma, we can derive the following results on the  bound $u $ for
in an $L^k$ space for any $k> 1$.

\begin{corollary}\label{lemma456ssdddddfgg30}
Let $(u ,v ,w )$ be a solution to \dref{7101.2x19x3sss189} on $(0,T_{max})$.
 If  $$\mu\geq\max_{s\geq1}\lambda_0^{\frac{1}{{{s}}+1}}(\chi+\xi\|w_0\|_{L^\infty(\Omega)}),$$
  then
 for all $k>1$,
there exists a positive constant $C$ which depends on $k$ such that 
\begin{equation}
\int_{\Omega}u ^{k}(x,t) \leq C ~~~\mbox{for all}~~ t\in(0,T_{max}).
\label{zjscz2.ddffrr5297x9xxcc6302222114}
\end{equation}
\end{corollary}
\begin{proof}
This directly results from Lemma \ref{qqqqlemma45630} and the fact that
$$\frac{\max_{s\geq1}\lambda_0^{\frac{1}{{{s}}+1}}(\chi+\xi\|w_0\|_{L^
\infty(\Omega)})}{(\max_{s\geq1}\lambda_0^{\frac{1}{{{s}}+1}}(\chi+\xi\|w_0\|_{L^\infty(\Omega)})-\mu)_{+}}=+\infty$$
by using
$\mu\geq\max_{s\geq1}\lambda_0^{\frac{1}{{{s}}+1}}(\chi+\xi\|w_0\|_{L^\infty(\Omega)}).$
\end{proof}
In the following, we always  assume that $$\mu<\max_{s\geq1}\lambda_0^{\frac{1}{{{s}}+1}}(\chi+\xi\|w_0\|_{L^\infty(\Omega)}),$$
since, case $\mu\geq\max_{s\geq1}\lambda_0^{\frac{1}{{{s}}+1}}(\chi+\xi\|w_0\|_{L^\infty(\Omega)})$ has been proved by  {Corollary} \ref{lemma456ssdddddfgg30}.
\begin{lemma}\label{lemma45630}
Let $(u ,v ,w )$ be a solution to \dref{7101.2x19x3sss189} on $(0,T_{max})$
and $\Omega\subset \mathbb{R}^N(N\geq1)$ be a bounded domain with smooth boundary.  Then for all $\beta>1,$ there exists $\kappa_0>0$ such that
\begin{equation}\label{hjui909klopji115}
\begin{array}{rl}
&\disp{\frac{1}{{2\beta}}\frac{d}{dt}\|\nabla v \|^{{{2\beta}}}_{L^{{2\beta}}(\Omega)}+\frac{(\beta-1)}{{\beta^2}}\disp\int_{\Omega}\left|\nabla |\nabla v |^{\beta}\right|^2}\\
&+\disp{\frac{1}{2}\disp\int_\Omega  |\nabla v |^{2\beta-2}|D^2v  |^2+\disp\int_{\Omega} |\nabla v |^{2\beta}{}}\\
\leq&\disp{\kappa_0\int_\Omega u ^2 |\nabla v |^{2\beta-2}+\kappa_0~~\mbox{for all}~~ t\in(0,T_{max}).}\\
\end{array}
\end{equation}
\end{lemma}
\begin{proof}
 Using that $\nabla  v \cdot\nabla\Delta {v}   = \frac{1}{2}\Delta |\nabla v |^2-|D^2v  |^2$, by a straightforward computation using the second
equation in \dref{7101.2x19x3sss189} and several integrations by parts, we find that
\begin{equation}
\begin{array}{rl}
&\disp{\frac{1}{{2\beta}}\frac{d}{dt} \|\nabla  v \|^{{{2\beta}}}_{L^{{2\beta}}(\Omega)}}\\
= &\disp{\disp\int_{\Omega} |\nabla v |^{2\beta-2}\nabla  v \cdot\nabla(\Delta  v  - v  +{u} )}
\\
=&\disp{\frac{1}{{2}}\disp\int_{\Omega} |\nabla v |^{2\beta-2}\Delta |\nabla v |^2-\disp\int_{\Omega} |\nabla v |^{2\beta-2}|D^2 v  |^2}\\
&-\disp\int_{\Omega} |\nabla v |^{2\beta}-\disp{\disp\int_\Omega   u  \nabla\cdot( |\nabla v |^{2\beta-2}\nabla  v  )}\\
=&\disp{-\frac{\beta-1}{{2}}\disp\int_{\Omega} |\nabla v |^{2\beta-4}\left|\nabla |\nabla v |^{2}\right|^2+\frac{1}{{2}}\disp\int_{\partial\Omega} |\nabla v |^{2\beta-2}\frac{\partial  |\nabla v |^{2}}{\partial\nu}-\disp\int_{\Omega} |\nabla v |^{2\beta}}\\
&-\disp{\disp\int_{\Omega} |\nabla v |^{2\beta-2}|D^2 v  |^2-\disp\int_\Omega {u} |\nabla v |^{2\beta-2}\Delta v  -\disp\int_\Omega   u  \nabla  v \cdot\nabla( |\nabla v |^{2\beta-2})}\\
=&\disp{-\frac{2(\beta-1)}{{\beta^2}}\disp\int_{\Omega}\left|\nabla |\nabla v |^{\beta}\right|^2+\frac{1}{{2}}\disp\int_{\partial\Omega} |\nabla v |^{2\beta-2}\frac{\partial  |\nabla v |^{2}}{\partial\nu}-\disp\int_{\Omega} |\nabla v |^{2\beta-2}|D^2 v  |^2}\\
&-\disp{\disp\int_\Omega {u}  |\nabla v |^{2\beta-2}\Delta  v  -\disp\int_\Omega   u  \nabla  v \cdot\nabla( |\nabla v |^{2\beta-2})-\disp\int_{\Omega} |\nabla v |^{2\beta}}\\
\end{array}
\label{cz2.5ghju48156}
\end{equation}
for all $t\in(0,T_{max})$.
Here, since $|\Delta v  | \leq\sqrt{N}|D^2v  |$, by the Young inequality, we can estimate
\begin{equation}
\begin{array}{rl}
\disp\int_\Omega {u}  |\nabla v |^{2\beta-2}\Delta  v \leq&\disp{\sqrt{N}\disp\int_\Omega {u}  |\nabla v |^{2\beta-2}|D^2v  |}
\\
\leq&\disp{\frac{1}{4}\disp\int_\Omega  |\nabla v |^{2\beta-2}|D^2v  |^2+N\disp\int_\Omega u ^2 |\nabla v |^{2\beta-2}}\\
\end{array}
\label{cz2.5ghju48hjuikl1}
\end{equation}
for all $t\in(0,T_{max})$. As moreover by the Cauchy--Schwarz inequality, we have
\begin{equation}
\begin{array}{rl}
-\disp\int_\Omega   u  \nabla  v \cdot\nabla( |\nabla v |^{2\beta-2})= &\disp{-(\beta-1)\disp\int_\Omega {u}  |\nabla v |^{2(\beta-2)}\nabla  v \cdot
\nabla |\nabla v |^{2}}\\
\leq &\disp{\frac{\beta-1}{8}\disp\int_{\Omega} |\nabla v |^{2\beta-4}\left|\nabla |\nabla v |^{2}\right|^2+2(\beta-1)
\disp\int_\Omega u ^2 |\nabla v |^{2\beta-2}}\\
\leq &\disp{\frac{(\beta-1)}{2{\beta^2}}\disp\int_{\Omega}\left|\nabla |\nabla v |^{\beta}\right|^2+2(\beta-1)
\disp\int_\Omega u ^2 |\nabla v |^{2\beta-2}.}\\
\end{array}
\label{cz2.5ghju4ghjuk81}
\end{equation}
Next we deal with the integration on $\partial\Omega$. We see from Lemma \ref{lemma41ffgg} that
\begin{equation}
\begin{array}{rl}
&\disp{\disp\int_{\partial\Omega}\frac{\partial |\nabla v |^2}{\partial\nu} |\nabla v |^{2\beta-2} }\\
\leq&\disp{C_\Omega\disp\int_{\partial\Omega} |\nabla v |^{2\beta} }\\
=&\disp{C_\Omega| |\nabla v |^{\beta}|^2_{L^2(\partial\Omega)}.}\\
\end{array}
\label{cz2.57151hhkkhhgg}
\end{equation}
Let us take $r\in(0,\frac{1}{2})$. By the embedding $W^{r+\frac{1}{2},2}(\Omega)\hookrightarrow L^2(\partial\Omega)$ is compact (see e.g. Haroske and Triebel \cite{Haroske}), we have
\begin{equation}
\begin{array}{rl}
&\disp{\| |\nabla v |^{\beta}\|^2_{L^2{(\partial\Omega})}\leq C_3\| |\nabla v |^{\beta}\|^2_{W^{r+\frac{1}{2},2}(\Omega)}.}\\
\end{array}
\label{cz2.57151}
\end{equation}
In order to apply Lemma \ref{lemma41ffgg} to the right-hand side of \dref{cz2.57151}, let us pick $a\in(0,1)$ satisfying
$$a=\frac{\frac{1}{2N}+\frac{\beta}{l}+\frac{\gamma}{N}-\frac{1}{2}}{\frac{1}{N}+\frac{\beta}{l}-\frac{1}{2}}.$$
Noting that $\gamma\in(0,\frac{1}{2})$ and $\beta>1$ imply that $\gamma+\frac{1}{2}\leq a<1$, we see from the fractional Gagliardo--Nirenberg inequality (Lemma \ref{lemma41ffgg}) and boundedness of $ |\nabla v |^l$ (see Lemma \ref{wsdelemma45}) that
\begin{equation}
\begin{array}{rl}
&\disp{| |\nabla v |^{\beta}|^2_{W^{r+\frac{1}{2},2}(\Omega)}}
\\
\leq&\disp{c_0|\nabla |\nabla v |^{\beta}|^a_{L^2(\Omega)}\| |\nabla v |^\beta\|^{1-a}_{L^{\frac{l}{\beta}}(\Omega)}+c'_0\| |\nabla v |^\beta\|_{L^{\frac{l}{\beta}}(\Omega)}}\\
\leq&\disp{C_4|\nabla |\nabla v |^{\beta}|^a_{L^2(\Omega)}+C_4.}\\
\end{array}
\label{vvggcz2.57151}
\end{equation}
Combining \dref{cz2.57151hhkkhhgg} and \dref{cz2.57151} with \dref{vvggcz2.57151}, we obtain
\begin{equation}
\begin{array}{rl}
&\disp{\disp\int_{\partial\Omega}\frac{\partial |\nabla v |^2}{\partial\nu} |\nabla v |^{2\beta-2} \leq C_5|\nabla |\nabla v |^{\beta}|^a_{L^2(\Omega)}+C_5.}\\
\end{array}
\label{cz2.57151hhkkhhggyyxx}
\end{equation}
Now, inserting \dref{cz2.5ghju4ghjuk81}--\dref{cz2.57151hhkkhhggyyxx} into \dref{cz2.5ghju48156} and using the Young inequality  we can get
\begin{equation}\label{dddffhjui909klopji115}
\begin{array}{rl}
&\disp{\frac{1}{{2\beta}}\frac{d}{dt}\|\nabla  v \|^{{{2\beta}}}_{L^{{2\beta}}(\Omega)}+\frac{3(\beta-1)}{4{\beta^2}}\disp\int_{\Omega}\left|\nabla |\nabla v |^{\beta}\right|^2+\frac{1}{2}\disp\int_\Omega  |\nabla v |^{2\beta-2}|D^2v  |^2+\disp\int_{\Omega} |\nabla v |^{2\beta}}\\
\leq&\disp{C_6\disp\int_\Omega u ^2 |\nabla v |^{2\beta-2}+C_6~~\mbox{for all}~~ t\in(0,T_{max})}\\
\end{array}
\end{equation}
by using the Young inequality.
\end{proof}

We proceed to establish the main step towards our boundedness proof. The following lemma can be used to improve our knowledge on integrability of $\nabla v $, provided
that $\mu>0$. Its repeated application will form the core of our regularity proof.
\begin{lemma}\label{lemma45630}
Let $(u ,v ,w )$ be a solution to \dref{7101.2x19x3sss189} on $(0,T_{max})$ and $\mu>0$.
Then for any $1<\gamma_0<\mu_*$, there
exists  $C > 0$  such that
\begin{equation}
\|\nabla v (\cdot, t)\|_{L^{{{\frac{(\gamma_0+1)(N+\gamma_0-1)}{N}}}}(\Omega)}\leq C ~~\mbox{for all}~~ t\in(0, T_{max}),
\label{zjscz2.5297x9ssd630xxy}
\end{equation}
where $\mu_*$ is given by  \dref{zjscz2.ddffrr5297x96ssddffdffggbh302222114}.
\end{lemma}
\begin{proof}
Let $\gamma_0$ and $\mu_*$ be same as Lemma \ref{qqqqlemma45630}.
For the above $1<\gamma_0<\mu_*$,
we choose $\beta=\frac{\gamma_0+1}{2}$ in \dref{hjui909klopji115}.
Then by using the Young inequality, we derive that for some positive constant $C_1$,
\begin{equation}\label{sssshjui909dddklopji115}
\begin{array}{rl}
\disp{\kappa_0\int_\Omega {u ^2} |\nabla {v }|^{2\beta-2}}
=&\disp{\kappa_0\int_\Omega {u ^2} |\nabla {v }|^{\gamma_0-1}}\\
\leq&\disp{\frac{1}{2}\int_\Omega|\nabla {v }|^{\gamma_0+1}+C_1\int_\Omega {u }^{\gamma_0+1}~~\mbox{for all}~~ t\in(0,T_{max}).}\\
\end{array}
\end{equation}
Here $\kappa_0$ is the same as \dref{hjui909klopji115}.
Inserting \dref{sssshjui909dddklopji115} into \dref{hjui909klopji115}, we conclude that there exists a positive constant $C_2$ such that
\begin{equation}\label{hjui909klopssddji115}
\begin{array}{rl}
&\disp{\frac{1}{{\gamma_0+1}}\frac{d}{dt}\|\nabla {v }\|^{{{\gamma_0+1}}}_{L^{{\gamma_0+1}}(\Omega)}+\frac{3(\frac{\gamma_0+1}{2}-1)}{{(\gamma_0+1)^2}}\disp\int_{\Omega}\left|\nabla |\nabla {v }|^{\frac{\gamma_0+1}{2}}\right|^2}\\
&\disp{+\frac{1}{2}\disp\int_\Omega  |\nabla {v }|^{\gamma_0-1}|D^2{v }|^2+\disp\frac{1}{2}\int_{\Omega} |\nabla {v }|^{\gamma_0+1}}\\
\leq&\disp{C_1\int_\Omega {u ^{\gamma_0+1}}+C_2~~\mbox{for all}~~ t\in(0,T_{max}),}\\
\end{array}
\end{equation}
which combined with \dref{4455zjscz2.ddffrr5297x96302222114} implies that
\begin{equation}
\int_{\Omega}|\nabla v |^{\gamma_0+1}(x,t)dx \leq C_3 ~~~\mbox{for all}~~ t\in(0,T_{max})
\label{zjscz2.ddffrr5297x96302sss222114}
\end{equation}
by an ODE comparison argument.
On the other hand, for any $\beta>1$,
it then follows from 
Lemma \ref{lemma41} that there exist positive constants $\kappa_{1}$ and $\kappa_{2}$ such that
\begin{equation}
\begin{array}{rl}
 \|\nabla {v }\|_{L^{2\beta+\frac{2(\gamma_0+1)}{N}}(\Omega)}^{2\beta+\frac{2(\gamma_0+1)}{N}}=&\disp{\| |\nabla {v }|^\beta\|_{L^{2+\frac{2(\gamma_0+1)}{\beta N}}(\Omega)}^{2+\frac{2(\gamma_0+1)}{N\beta}}}
\\
\leq&\disp{\kappa_{1}(\|\nabla |\nabla {v }|^\beta\|_{L^2(\Omega)}^{2}\| |\nabla {v }|^\beta\|_{L^\frac{\gamma_0+1}{\beta}(\Omega)}^{\frac{2(\gamma_0+1)}{N\beta}}+\|   |\nabla {v }|^\beta\|_{L^\frac{\gamma_0+1}{\beta}
(\Omega)}^{2+\frac{2(\gamma_0+1)}{N\beta}})}\\
\leq&\disp{\kappa_{2}(\|\nabla |\nabla {v }|^\beta\|_{L^2(\Omega)}^{2}+1)}\\
\end{array}
\label{9999cz2.563022222ikopl2gg66}
\end{equation}
by using \dref{zjscz2.ddffrr5297x96302sss222114}. Next, picking $\beta=\frac{(\gamma_0+1)(N+\gamma_0-1)}{2N}$ in \dref{hjui909klopji115}, then $\beta>1$, so that,
by \dref{hjui909klopji115}, we derive that 
\begin{equation}\label{hjeeui909sddfghhklopji115}
\begin{array}{rl}
&\disp{\frac{1}{{\frac{(\gamma_0+1)(N+\gamma_0-1)}{N}}}\frac{d}{dt}\|\nabla {v }\|^{{{\frac{(\gamma_0+1)(N+\gamma_0-1)}{N}}}}_{L^{{\frac{(\gamma_0+1)(N+\gamma_0-1)}{N}}}(\Omega)}}\\
&\disp{+\frac{3({\frac{(\gamma_0+1)(N+\gamma_0-1)}{2N}}-1)}{{({\frac{(\gamma_0+1)(N+\gamma_0-1)}{N}})^2}}\disp\int_{\Omega}\left|\nabla |\nabla {v }|^{{\frac{(\gamma_0+1)(N+\gamma_0-1)}{2N}}}\right|^2}\\
&\disp{+\frac{1}{2}\disp\int_\Omega  |\nabla {v }|^{{\frac{(\gamma_0+1)(N+\gamma_0-1)}{N}}-2}|D^2{v }|^2+\disp\int_{\Omega} |\nabla {v }|^{\frac{(\gamma_0+1)(N+\gamma_0-1)}{N}}}\\
\leq&\disp{\kappa_0\disp\int_\Omega {u ^2} |\nabla {v }|^{{\frac{(\gamma_0+1)(N+\gamma_0-1)}{N}}-2}+\kappa_0}\\
\leq&\disp{\frac{{\frac{(\gamma_0+1)(N+\gamma_0-1)}{2N}}-1}{{({\frac{(\gamma_0+1)(N+\gamma_0-1)}{N}})^2\kappa_2}}\|\nabla {v }\|_{L^{2\beta+\frac{2(\gamma_0+1)}{N}}(\Omega)}^{2\beta+\frac{2(\gamma_0+1)}{N}}+C_{4}\disp\int_\Omega {u ^{\gamma_0+1}}+C_{5}~~\mbox{for all}~~ t\in(0,T_{max}),}\\
\end{array}
\end{equation}
where $\kappa_2$ is the same as \dref{9999cz2.563022222ikopl2gg66}. Therefore, collecting \dref{9999cz2.563022222ikopl2gg66} and \dref{hjeeui909sddfghhklopji115}, we have
\begin{equation}\label{hjeeui9ssd09sddfghhklopji115}
\begin{array}{rl}
&\disp{\frac{1}{{\frac{(\gamma_0+1)(N+\gamma_0-1)}{N}}}\frac{d}{dt}\|\nabla {v }\|^{{{\frac{(\gamma_0+1)(N+\gamma_0-1)}{N}}}}_{L^{{\frac{(\gamma_0+1)(N+\gamma_0-1)}{N}}}(\Omega)}}\\
&\disp{+\frac{2({\frac{(\gamma_0+1)(N+\gamma_0-1)}{2N}}-1)}{{({\frac{(\gamma_0+1)(N+\gamma_0-1)}{N}})^2}}\disp\int_{\Omega}\left|\nabla |\nabla {v }|^{{\frac{(\gamma_0+1)(N+\gamma_0-1)}{2N}}}\right|^2}\\
&\disp{+\frac{1}{2}\disp\int_\Omega  |\nabla {v }|^{{\frac{(\gamma_0+1)(N+\gamma_0-1)}{N}}-2}|D^2{v }|^2+\disp\int_{\Omega} |\nabla {v }|^{\frac{(\gamma_0+1)(N+\gamma_0-1)}{N}}}\\
\leq&\disp{C_{4}\disp\int_\Omega {u ^{\gamma_0+1}}+C_{5}~~\mbox{for all}~~ t\in(0,T_{max}),}\\
\end{array}
\end{equation}
therefore, in view of \dref{4455zjscz2.ddffrr5297x96302222114},  by using an ODE comparison argument again, we have
\begin{equation}
\begin{array}{rl}
 \|\nabla {v }\|_{L^{{{\frac{(\gamma_0+1)(N+\gamma_0-1)}{N}}}}(\Omega)}\leq&\disp{C_{6}~~\mbox{for all}~~ t\in(0,T_{max})}\\
\end{array}
\label{9999cz2.563022222ikopl2ssddgg66}
\end{equation}
with some positive constant $C_{6}$, which yields \dref{zjscz2.5297x9ssd630xxy}, and hence completes the proof.
\end{proof}

\begin{lemma}\label{pplemma45630}
Let $\Omega\subset \mathbb{R}^N(N\geq1)$ be a bounded domain with smooth boundary.  Then for all $\beta> 1$ and $k>1$, the solution of  \dref{7101.2x19x3sss189} from Lemma \ref{lemma70} satisfies
\begin{equation}\label{1234hjui909klopji115}
\begin{array}{rl}
&\disp{\frac{d}{dt}(\frac{1}{k}\|u \|^{k}_{L^k(\Omega)}+\frac{1}{{2\beta}}\|\nabla {v }\|^{{{2\beta}}}_{L^{{2\beta}}(\Omega)})+\frac{3(\beta-1)}{4{\beta^2}}\disp\int_{\Omega}\left|\nabla |\nabla {v }|^{\beta}\right|^2+\frac{\mu}{2}\int_{\Omega}u ^{k+1}}\\
&+\disp{\frac{1}{2}\disp\int_\Omega  |\nabla {v }|^{2\beta-2}|D^2{v }|^2+\disp\int_{\Omega} |\nabla {v }|^{2\beta}+\frac{(k-1)m}{4}\int_{\Omega}u ^{m+k-3}|\nabla u |^2{}}\\
\leq&\disp{C(\disp\frac{\chi^2(k-1)}{2C_D}\int_{\Omega}u ^{k+1-m}|\nabla v |^2+\int_\Omega u ^2 |\nabla {v }|^{2\beta-2}+\int_\Omega v ^{k+1})+C,}\\
\end{array}
\end{equation}
where $C$ is a positive constant.
\end{lemma}
\begin{proof}
Collecting Lemma \ref{lemmssdda45630} and Lemma \ref{lemma45630},  we can derive \dref{1234hjui909klopji115} by using the Young inequality.
\end{proof}

We next plan to estimate the right-hand sides in the above inequalities appropriately by using  a priori information provided by Lemma \ref{lemma45630} and Lemma \ref{wsdelemma45}. Here the
following lemma will will play an important role in making
efficient use of the known ${L^{{{\frac{(\gamma_0+1)(N+\gamma_0-1)}{N}}}}(\Omega)}$ bound for $\nabla v $. The following lemma provides some elementary material that will be essential to our
bootstrap procedure.

\begin{lemma}\label{lsssemma456302ssd23116}
Let 
\begin{equation}\label{dcfvwsdddffgg7101.2x19x318jkl}
\tilde{H}(y)=\begin{array}{ll}
\frac{2N^2}{N^2+{{{(y+1)(N+y-1)}}}}-[1+\frac{[N^2-{{{(y+1)
(N+y-1)}}}]y}
{N{{{(y+1)
(N+y-1)}}}}]
 \end{array}
\end{equation}
with ${{{(y+1)
(N+y-1)}}}>N^2$,
for any $y>1$ and $N\geq2.$
Then
we have
\begin{equation}\label{dcfvdddffffggffffwsdddffgg7101.2x19x318jkl}
\min_{y>1}\tilde{H}(y)\geq0.
\end{equation}
\end{lemma}
\begin{proof}
It is easy to verify that
$N^2<{{{(y+1)
(N+y-1)}}}$ and $y>1$ and $N\geq2$
implies that
\begin{equation}\label{dcfvdffffwsdddffgg7101.2x19x318jkl}y>\frac{-N+\sqrt{5N^2-4N+4}}{2}=\frac{-N+\sqrt{4N^2+(N-2)^2}}{2}\geq\frac{N}{2}.
\end{equation}
On the other hand, by some basic calculation, one has
\begin{equation}\label{dcfvdfssddfffwsdddffgg7101.2x19x318jkl}\begin{array}{ll}
&\frac{2N^2}{N^2+{{{(y+1)(N+y-1)}}}}-[1+\frac{[N^2-{{{(y+1)
(N+y-1)}}}]y}
{N{{{(y+1)
(N+y-1)}}}}]\\
=&[{{{(y+1)(N+y-1)}}}-N^2][\frac{y}
{N{{{(y+1)
(N+y-1)}}}}-\frac{1}{N^2+{{{(y+1)(N+y-1)}}}}]\\
=&\frac{[{{{(y+1)(N+y-1)}}}-N^2]}{N{{{(y+1)
(N+y-1)[N^2+{{{(y+1)(N+y-1)}}}}}]}}h_1(y)\\
\end{array}
\end{equation}
%
%
%
with
\begin{equation}\label{dcfvdfsdfffddsddfffwsdddffgg7101.2x19x318jkl}\begin{array}{rl}
h_1(y)=&[yN^2+y(y+1)(y+N-1)-N(y+1)(N+y-1)]\\
=&y^3+(N-1)y-N^2+N.\\
\end{array}
\end{equation}
Now, by some basic calculation, one has,
$$\begin{array}{rl}
h'_1(y)=&3y^2+(N-1)>0.\\
\end{array}$$
by using $N\geq2$ and $y>1$.
Therefore, by \dref{dcfvdffffwsdddffgg7101.2x19x318jkl}, we have
\begin{equation}\label{dcfvdffddsddffffgg710x19x318jkl}\begin{array}{rl}
h_1(y)\geq&h_1(\frac{N}{2})\\
=&\frac{N^3}{8}-\frac{N^2}{2}+\frac{N}{2}\\
=:&\tilde{h}_1(N).\\
\end{array}
\end{equation}
Since, $\tilde{h}'_1(N)=\frac{3N^2}{8}-N+\frac{1}{2}>0$ by using $N\geq2.$
Thus, $\tilde{h}_1(N)>\tilde{h}_1(2)=0,$
so that, inserting \dref{dcfvdfsdfffddsddfffwsdddffgg7101.2x19x318jkl}--\dref{dcfvdffddsddffffgg710x19x318jkl} into \dref{dcfvdfssddfffwsdddffgg7101.2x19x318jkl} and ${{{(y+1)
(N+y-1)}}}>N^2$, we obtain that
\begin{equation}\label{dcfvdffddddfffsddsdddffffgg710x19x318jkl}\frac{2N^2}{N^2+{{{(y+1)(N+y-1)}}}}>[1+\frac{[N^2-{{{(y+1)
(N+y-1)}}}]y}
{N{{{(y+1)
(N+y-1)}}}}].
\end{equation}
\end{proof}

Now,
we can make use of Lemma \ref{ssdeedrfe116lemma70hhjj} as well as Lemma  \ref{lemma41}  and Lemma \ref{lemma45630} to estimate the integrals on the right-hand sides of \dref{hjui909klopji115} and \dref{cz2.5xx1jjjj} (or \dref{1234hjui909klopji115}). To this end, we will
establish bounds for $\int_\Omega u^k  dx$ with any $k> 1$ by Lemma \ref{ssdeedrfe116lemma70hhjj} as well as Lemma  \ref{lemma41}  and Lemma \ref{lsssemma456302ssd23116}.

\begin{lemma}\label{lemma456ssddfff30}
Assume that  $m>\frac{2N}{N+\gamma_*}$ with $N=2$,
 where
\begin{equation}\label{xeerr1.731426677gg}
\gamma_*={{{\frac{(\mu_{*}+1)(N+\mu_{*}-1)}{N}}}}
 \end{equation}
and  $\mu_*$ is same as   \dref{zjscz2.ddffrr5297x96ssddffdffggbh302222114}.
  Then for all  $k>1$, there
 exists
  $C > 0$  such that 
\begin{equation}
\|u (\cdot, t)\|_{L^k(\Omega)}\leq C ~~\mbox{for all}~~ t\in(0, T_{max}).
\label{111zjscz2.5297x9630xxy}
\end{equation}
\end{lemma}
\begin{proof}
Next, due to \dref{xeerr1.731426677gg} and $N=2$ and $m>\frac{2N}{N+\gamma_*}$ implies that
\begin{equation}m>\frac{8}{4+[\mu_*+1]^2},
\label{zjsssdddcz2.5297x9ssd63ss0xxy}
\end{equation}
where $\mu_*$ is given by  \dref{zjscz2.ddffrr5297x96ssddffdffggbh302222114}.
Now, in view of $\mu>0$ implies that
$$\mu_*>1$$
and
$$4<(\mu_*+1)^2,$$
therefore,
employing Lemma \ref{lsssemma456302ssd23116}, we have
\begin{equation}m>1+\frac{[1-\frac{(\mu_*+1)^2}{4}]\mu_*}
{2\times\frac{(\mu_*+1)^2}{4}}.
\label{zjsssdddcz2.5297x9ssssddd63ss0xxy}
\end{equation}
Thus, we may choose
 $ q_0\in(1,\mu_*)$
which is close to $\mu_*$ 
 such that
\begin{equation}
\begin{array}{rl}
m>&1+\frac{[1-\frac{(q_0+1)^2}{4}]q_0}
{2\times\frac{(q_0+1)^2}{4}}.\\
\end{array}\label{zjsssdddcz2.5297x9ssssddd63ss0ssddxxy}
\end{equation}
%
Next, observing that ${{{{\frac{(q_0+1)^2}{4}}}}}\in(1,\frac{(\mu_*+1)^2}{4})$, thus,
in light of  Lemma \ref{lemma45630}, 
 we derive that there exists a positive constant $C_1$ such that
\begin{equation}
\|\nabla v (\cdot, t)\|_{L^{2p_0}(\Omega)}\leq C_1 ~~\mbox{for all}~~ t\in(0, T_{max}),
\label{zjscz2.5297x9ssd63ss0xxy}
\end{equation}
where \begin{equation}
\label{zjscz2.5297x9ssd63ssdddss0xxy}p_0={{{{\frac{(q_0+1)^2}{4}}}}}>1
\end{equation}
 by using $q_0>1.$ 

Now, choosing $k> \max\{3,|1-m|+q_0\frac{p_0-1}{p_0}\}$ in  \dref{cz2.5xx1jjjj}, then, we have
\begin{equation}
\begin{array}{rl}
&\disp{\frac{1}{k}\frac{d}{dt}\| u  \|^{k}_{L^k(\Omega)}+\frac{(k-1) C_D }{2}\int_{\Omega} u  ^{m+k-3}|\nabla u  |^2{}+
\mu\int_{\Omega} u  ^{k+1}{}}
\\
\leq&\disp{\frac{\chi^2(k-1)} {2 C_D }\int_{\Omega} u  ^{k+1-m}|\nabla v  |^2+
C_2\int_\Omega  u ^{k}(v +1)}\\
\end{array}
\label{vgbhnsxcdvfddfcz2.5xxdsddddddf1jjsddffjj}
\end{equation}
with $C_2=\max\{\xi\|w_0\|_{L^\infty(\Omega)},\mu+\kappa\xi\}$.
 We estimate the rightmost integral by means of the H\"{o}lder inequality according to
\begin{equation}
\begin{array}{rl}
&\disp{\frac{ \chi^2(k-1)}{2C_D} \disp\int_\Omega u ^{k+1-m } |\nabla v  |^2}\\
&\leq\disp{ \frac{ \chi^2(k-1)}{2C_D}\left(\disp\int_\Omega u ^{\frac{p_0}{p_0-1}(k+1-m )}\right)^{\frac{p_0-1}{p_0}}\left(\disp\int_\Omega |\nabla v  |^{2p_0}\right)^{\frac{1}{p_0}}}\\
&\leq\disp{ \frac{C_3 \chi^2(k-1)}{2C_D}\|   u ^{\frac{k+m-1}{2}}\|^{\frac{2(k+1-m )}{k+m-1}}_{L^{\frac{2p_0(k+1-m )}{(p_0-1)(k+m-1)}}(\Omega)}}\\
\end{array}
\label{cz2.57151hhkkhhhjukildrfthjjhhhhh}
\end{equation}
by using \dref{zjscz2.5297x9ssd63ss0xxy},
where $C_3>0$. 
Since,  $k> |1-m|+q_0\frac{p_0-1}{p_0},$
we have
$$\frac{q_0}{k+m-1}\leq{\frac{p_0(k+1-m )}{(p_0-1)(k+m-1)}}<+\infty,$$
so that, the Gagliardo-Nirenberg inequality (Lemma \ref{lemma41}) indicates
that
\begin{equation}
\begin{array}{rl}
&\disp{\frac{ \chi^2(k-1)}{2C_D}\|   u ^{\frac{k+m-1}{2}}\|
^{\frac{2(k+1-m )}{k+m-1}}_{L^{\frac{2p_0(k+1-m )}{(p_0-1)(k+m-1)}}(\Omega)}}
\\
\leq&\disp{C_{4}(\|\nabla    u ^{\frac{k+m-1}{2}}\|_{L^2(\Omega)}^{\frac{2(k+1-m)}{k+m-1}-\frac{2(p_0-1)q_0}{p_0(k+m-1)}}\|  u ^{\frac{k+m-1}{2}}\|_{L^\frac{2q_0}{k+m-1}(\Omega)}^{\frac{2(p_0-1)q_0}{p_0(k+m-1)}}+\|   u ^{\frac{k+m-1}{2}}\|_{L^\frac{2q_0}{k+m-1}(\Omega)}^{\frac{2(k+1-m )}{k+m-1}})}\\
\leq&\disp{C_{5}(\|\nabla    u ^{\frac{k+m-1}{2}}\|_{L^2(\Omega)}^{\frac{2(k+1-m)}{k+m-1}-\frac{2(p_0-1)q_0}{p_0(k+m-1)}}+1)}\\
\end{array}
\label{cz2.563022222ikopl2sdfg44}
\end{equation}
with some positive constants $C_{4}$ as well as $ C_{5}$, where $q_0$ is the same as \dref{zjsssdddcz2.5297x9ssssddd63ss0ssddxxy}.
Here we have used   $L^1(\Omega)$ boundedness for $u $ (see Lemma \dref{wsdelemma45}).
Due to \dref{zjsssdddcz2.5297x9ssssddd63ss0ssddxxy}, one has
$${\frac{2(k+1-m)}{k+m-1}-\frac{2(p_0-1)q_0}{p_0(k+m-1)}}<2,$$
so that, applying the Young inequality implies that there exists a positive constant $C_7$ such that
\begin{equation}
\begin{array}{rl}
&\disp{\frac{1}{k}\frac{d}{dt}\|u \|^{k}_{L^k(\Omega)}+\frac{ (k-1)  C_D }{4}\int_{\Omega} u ^{m+k-3}|\nabla u  |^2{}+
\frac{\mu}{2}\int_{\Omega} u ^{k+1}{}}
\\
\leq&\disp{
C_2\int_\Omega  u ^{k}(v +1)+C_7~~\mbox{for all}~~ t\in(0, T_{max}),}\\
\end{array}
\label{vgbhnsxcdvfdssddfcz2.5xxdddf1jjjj}
\end{equation}
which together with the Young inequality implies that
\begin{equation}
\begin{array}{rl}
&\disp{\frac{1}{k}\frac{d}{dt}\|u \|^{k}_{L^k(\Omega)}+\frac{ (k-1)  C_D }{4}\int_{\Omega} u ^{m+k-3}|\nabla u  |^2{}+
\mu\int_{\Omega} u ^{k+1}{}}
\\
\leq&\disp{
C_9\int_\Omega  (v +1)^k+C_8~~\mbox{for all}~~ t\in(0, T_{max})}\\
\end{array}
\label{vgbhnsxcdvfdssddfcz2.5xxdddf1ssddjjjj}
\end{equation}
for some positive constants  $C_8$ and $C_9.$ Now, in view of  ${2p_0}>2$ (see \dref{zjscz2.5297x9ssd63ssdddss0xxy})and $N=2$, then by Sobolev imbedding theorems, we derive from \dref{zjscz2.5297x9ssd63ss0xxy} that
$$\|v \|_{L^\infty(\Omega)}\leq C_{10}\|\nabla v \|_{L^{2p_0}(\Omega)},$$
so that, combined with \dref{vgbhnsxcdvfdssddfcz2.5xxdddf1ssddjjjj} implies that
\begin{equation}
\begin{array}{rl}
&\disp{\frac{1}{k}\frac{d}{dt}\|u \|^{k}_{L^k(\Omega)}+\frac{ (k-1)  C_D }{4}\int_{\Omega} u ^{m+k-3}|\nabla u  |^2{}+
\mu\int_{\Omega} u ^{k+1}{}}
\\
\leq&\disp{
C_{11}~~\mbox{for all}~~ t\in(0, T_{max}),}\\
\end{array}
\label{vgbhnsxcdvfdssddfcz2.5ssddxxdddf1ssddjjjj}
\end{equation}
whereas a standard ODE comparison argument shows that \dref{111zjscz2.5297x9630xxy} holds.
\end{proof}

\begin{lemma}\label{lemma456ssddfff30}
Let $\Omega\subset \mathbb{R}^N(N\geq1)$ be a bounded domain with smooth boundary. Furthermore,
assume that $m>\frac{2N}{N+\gamma_*}$ with $N\geq1$,
 where $\gamma_*$ is given by \dref{xeerr1.731426677gg}. If
 \begin{equation}{{{\frac{(\mu_*+1)
(N+\mu_*-1)}{2N}}}}>\frac{N}{2},
\label{zjscz2.529df763xy}
\end{equation} then
 for any $k>1$, there  exists a positive constant  $C$  such that
\begin{equation}
\|u (\cdot, t)\|_{L^k(\Omega)}\leq C ~~\mbox{for all}~~ t\in(0, T_{max}),
\label{2222zjscz2.5297x9630xxy}
\end{equation}
where $\mu_*$ is given by  \dref{zjscz2.ddffrr5297x96ssddffdffggbh302222114}.
\end{lemma}
\begin{proof}
%
Due to
$$m>\frac{2N}{N+{{{\frac{(\mu_*+1)(N+\mu_*-1)}{N}}}}}.$$
Now, in view of $\mu>0$ implies that
$$\mu_*>1,$$
which together with  Lemma \ref{lsssemma456302ssd23116} results in
\begin{equation}m>1+\frac{[N-2{{{\frac{(\mu_*+1)
(N+\mu_*-1)}{2N}}}}]\mu_*}
{2N\times{{{\frac{(\mu_*+1)
(N+\mu_*-1)}{2N}}}}}.
\label{zjsssdddcz2.5297x9ssssddd63ss0xxy}
\end{equation}
Thus by \dref{zjscz2.529df763xy}, we may choose $q_{0,*}\in(1,\mu_*)$ 
which is close to $\mu_*$ 
 such that
\begin{equation}
\begin{array}{rl}
m>&1+\frac{[N-2{{{\frac{(q_{0,*}+1)
(N+q_{0,*}-1)}{2N}}}}]q_{0,*}}
{2N\times{{{\frac{(q_{0,*}+1)
(N+q_{0,*}-1)}{2N}}}}}\\\
\end{array}\label{zjsssdd78888dcz2.5297x9ssssddd63ss0ssddxxy}
\end{equation}
and
 \begin{equation}{{{\frac{(q_{0,*}+1)
(N+q_{0,*}-1)}{2N}}}}>\frac{N}{2}.
\label{zjscz2.rrtt529df763xy}
\end{equation}
%
Next, observing that $\frac{(q_{0,*}+1)
(N+q_{0,*}-1)}{2N}\in(1,{{{\frac{(\mu_*+1)
(N+\mu_*-1)}{2N}}}})$, thus,
in light of  Lemma \ref{lemma45630}, 
 we derive that there exists a positive constant $C_1$ such that
\begin{equation}
\|\nabla v (\cdot, t)\|_{L^{2p_{0,*}}(\Omega)}\leq C_1 ~~\mbox{for all}~~ t\in(0, T_{max}),
\label{zjscz2dfffg.5297x9ssd63ss0xxy}
\end{equation}
where \begin{equation}
\label{zjscz2.5297x9ssd63ssdddss0xxy}p_{0,*}=\frac{(q_{0,*}+1)
(N+q_{0,*}-1)}{2N}>\frac{N}{2}.
\end{equation}

Now, choosing $k> \max\{N+1,|1-m|+q_{0,*}\frac{p_{0,*}-1}{p_{0,*}},\frac{1-m}{2p_{0,*}-N}(2p_{0,*}N-N-2p_{0,*})\}$ in  \dref{cz2.5xx1jjjj}, then, we have
\begin{equation}
\begin{array}{rl}
&\disp{\frac{1}{k}\frac{d}{dt}\|u \|^{k}_{L^k(\Omega)}+\frac{ (k-1)  C_D }{2}\int_{\Omega} u ^{m+k-3}|\nabla u  |^2{}+
\frac{\mu}{2}\int_{\Omega} u ^{k+1}{}}
\\
\leq&\disp{\frac{\chi^2(k-1)}{2 m}\int_{\Omega} u ^{k+1-m}|\nabla v  |^2+
C_2\int_\Omega  u ^{k}(v +1)~~\mbox{for all}~~ t\in(0, T_{max})}\\
\end{array}
\label{vgbhnsxcdvfddfcz2.5xxdddf1jjjj}
\end{equation}
with $C_2=\max\{\xi\|w_0\|_{L^\infty(\Omega)},\mu+\kappa\xi\}$.
 According to the estimate of $\nabla v $ in \dref{zjscz2dfffg.5297x9ssd63ss0xxy}  along with  the  H\"{o}lder inequality we
derive that  there exists a positive constant $C_3>0$ such that
\begin{equation}
\begin{array}{rl}
&\disp{\frac{ \chi^2(k-1)}{2C_D} \disp\int_\Omega u ^{k+1-m } |\nabla v  |^2}\\
&\leq\disp{ \frac{ \chi^2(k-1)}{2C_D}\left(\disp\int_\Omega u ^{\frac{p_{0,*}}{p_{0,*}-1}(k+1-m )}\right)^{\frac{p_{0,*}-1}{p_{0,*}}}\left(\disp\int_\Omega |\nabla v  |^{2p_{0,*}}\right)^{\frac{1}{p_{0,*}}}}\\
&\leq\disp{ \frac{C_3 \chi^2(k-1)}{2C_D}\|   u ^{\frac{k+m-1}{2}}\|^{\frac{2(k+1-m )}{k+m-1}}_{L^{\frac{2p_{0,*}(k+1-m )}{(p_{0,*}-1)(k+m-1)}}(\Omega)}
.}\\
\end{array}
\label{cz2.57151hhkkhhhjukildrfthjjhhhhh}
\end{equation}
In view of  $k> \max\{|1-m|+q_{0,*}\frac{p_{0,*}-1}{p_{0,*}},\frac{1-m}{2p_{0,*}-N}(2p_{0,*}N-N-2p_{0,*})\}$,
we have
$$\frac{q_{0,*}}{k+m-1}\leq{\frac{p_{0,*}(k+1-m )}{(p_{0,*}-1)(k+m-1)}}<\frac{N}{N-2}.$$
An application of the Gagliardo-Nirenberg inequality (see Lemma \ref{lemma41}) implies that for some positive constants $C_{4}$ as well as $ C_{5}$ such that
\begin{equation}
\begin{array}{rl}
&\disp{\frac{ \chi^2(k-1)}{2C_D}\|   u ^{\frac{k+m-1}{2}}\|
^{\frac{2(k+1-m )}{k+m-1}}_{L^{\frac{2p_{0,*}(k+1-m )}{(p_{0,*}-1)(k+m-1)}}(\Omega)}}
\\
\leq&\disp{C_{4}(\|\nabla    u ^{\frac{k+m-1}{2}}\|_{L^2(\Omega)}^{2\frac{N(k+1-m-q_{0,*}+\frac{q_{0,*}}{p_{0,*}})}{N(k+m-1-q_{0,*})+2q_{0,*}}}\|  u ^{\frac{k+m-1}{2}}\|_{L^\frac{2q_{0,*}}{k+m-1}(\Omega)}^{\frac{2(k+1-m )}{k+m-1}-2\frac{N(k+1-m-q_{0,*}+\frac{q_{0,*}}{p_{0,*}})}{N(k+m-1-q_{0,*})+2q_{0,*}}}+\|   u ^{\frac{k+m-1}{2}}\|_{L^\frac{2q_{0,*}}{k+m-1}(\Omega)}^{\frac{2(k+1-m )}{k+m-1}})}\\
\leq&\disp{C_{5}(\|\nabla    u ^{\frac{k+m-1}{2}}\|_{L^2(\Omega)}^{2\frac{N(k+1-m-q_{0,*}+\frac{q_{0,*}}{p_{0,*}})}{N(k+m-1-q_{0,*})+2q_{0,*}}}+1)}\\
\end{array}
\label{cz2.563022222ikopl2sdfg44}
\end{equation}
by using Lemma \ref{wsdelemma45}, where $q_{0,*}$ is the same as \dref{zjsssdd78888dcz2.5297x9ssssddd63ss0ssddxxy}.
Due to \dref{zjsssdd78888dcz2.5297x9ssssddd63ss0ssddxxy}, one has
$${2\frac{N(k+1-m-q_{0,*}+\frac{q_{0,*}}{p_{0,*}})}{N(k+m-1-q_{0,*})+2q_{0,*}}}<2,$$
so that, applying the Young inequality implies that
\begin{equation}
\begin{array}{rl}
&\disp{\frac{1}{k}\frac{d}{dt}\|u \|^{k}_{L^k(\Omega)}+\frac{ (k-1)  C_D }{4}\int_{\Omega} u ^{m+k-3}|\nabla u  |^2{}+
\frac{\mu}{2}\int_{\Omega} u ^{k+1}{}}
\\
\leq&\disp{
C_2\int_\Omega  u ^{k}(v +1)+C_7~~\mbox{for all}~~ t\in(0, T_{max}),}\\
\end{array}
\label{vgbhnsxcdvfds7888sddfcz2.5xxdddf1jjjj}
\end{equation}
which together with the Young inequality again yields to 
\begin{equation}
\begin{array}{rl}
&\disp{\frac{1}{k}\frac{d}{dt}\|u \|^{k}_{L^k(\Omega)}+\frac{ (k-1)  C_D }{4}\int_{\Omega} u ^{m+k-3}|\nabla u  |^2{}+
\mu\int_{\Omega} u ^{k+1}{}}
\\
\leq&\disp{
C_9\int_\Omega  (v +1)^k+C_8~~\mbox{for all}~~ t\in(0, T_{max})}\\
\end{array}
\label{vgbhnsxcdvfdssddfcz2jjjj.5xxdddf1ssddjjjj}
\end{equation}
and some positive constants  $C_8$ and $C_9.$ Now, in view of  ${2p_0}>N$ (see \dref{zjscz2.5297x9ssd63ssdddss0xxy}), then by Sobolev imbedding theorems, we derive from \dref{zjscz2dfffg.5297x9ssd63ss0xxy} that
$$\|v \|_{L^\infty(\Omega)}\leq C_{10}\|\nabla v \|_{L^{2p_0}(\Omega)},$$
so that, combined with \dref{vgbhnsxcdvfdssddfcz2jjjj.5xxdddf1ssddjjjj} implies that
\begin{equation}
\begin{array}{rl}
&\disp{\frac{1}{k}\frac{d}{dt}\|u \|^{k}_{L^k(\Omega)}+\frac{ (k-1)  C_D }{4}\int_{\Omega} u ^{m+k-3}|\nabla u  |^2{}+
\mu\int_{\Omega} u ^{k+1}{}}
\\
\leq&\disp{
C_{11}~~\mbox{for all}~~ t\in(0, T_{max}),}\\
\end{array}
\label{vgbhnsxcdvfdssddfcz2.5ssddxxdddf1ssddjjjj}
\end{equation}
whereas a standard ODE comparison argument shows that \dref{2222zjscz2.5297x9630xxy} holds.
\end{proof}

\begin{lemma}\label{lemmaddff45ssdd630}
Let $\Omega\subset \mathbb{R}^N(N\neq2)$ be a bounded domain with smooth boundary. Furthermore,
assume that $m>\frac{2N}{N+\gamma_*}$ with $N\geq1$,
 where $\gamma_*$ is given by \dref{xeerr1.731426677gg}. If
 \begin{equation}{{{\frac{(\mu_*+1)
(N+\mu_*-1)}{2N}}}}\leq\frac{N}{2},
\label{zjscz2.529df763xddfffy}
\end{equation} then
 for any $k>1$, there  exists a positive constant  $C$  such that
\begin{equation}
\|u (\cdot, t)\|_{L^k(\Omega)}\leq C ~~\mbox{for all}~~ t\in(0, T_{max}).
\label{zjscz2.529ddff7x9630xxy}
\end{equation}
\end{lemma}
\begin{proof}

Let $$\bar{\beta}=\max\{\frac{4N{{{\frac{(\mu_*+1)(N+\frac{\chi\max\{1,\lambda_0\}}
{(\chi\max\{1,\lambda_0\}-\mu)_{+}}-1)}{N}}}}}{N+{{{\frac{(\mu_*+1)(N+\frac{\chi\max\{1,\lambda_0\}}
{(\chi\max\{1,\lambda_0\}-\mu)_{+}}-1)}{N}}}}},2-\frac{2}{N}-m,16,8N+2\}.$$
Due to
$$m>\frac{2N}{N+{{{\frac{(\mu_*+1)(N+\frac{\chi\max\{1,\lambda_0\}}
{(\chi\max\{1,\lambda_0\}-\mu)_{+}}-1)}{N}}}}},$$
as well as \dref{zjscz2.529df763xddfffy} and Lemma \ref{lemma45630},
%
we may choose $q_{0,***}\in(1,\mu_*)$ 
which is close to $\mu_*$ 
 such that
%
\begin{equation}
\begin{array}{rl}
&\frac{2N(m -1)q_{0,***}}{N-2q_{0,***}}(\frac{2}{N}-1+\frac{\beta}{q_{0,***}})+2-m-\frac{2}{N}\\
&>\disp{ \frac{N(1-\frac{4}{\beta})}{1+\frac{N}{2q_{0,***}}-\frac{4N}{\beta}}(\frac{2}{N}-1+\frac{\beta}{q_{0,***}})+2-m-\frac{2}{N}}\\
\end{array}
\label{sxcdcz2.57151hhkkhhhjukildrfthjjhhhhh}
\end{equation}
for  all $\beta\geq\bar{\beta}.$
Therefore, we can choose
\begin{equation}
k\in( \frac{N(1-\frac{4}{\beta})}{1+\frac{N}{2q_{0,***}}-\frac{4N}{\beta}}(\frac{2}{N}-1+\frac{\beta}{q_{0,***}})+2-m-\frac{2}{N},\frac{2N(m -1)q_{0,***}}{N-2q_{0,***}}(\frac{2}{N}-1+\frac{\beta}{q_{0,***}})+2-m-\frac{2}{N}).
\label{dcfvzsxcvgcz2.57151hhkkhhhjukildrfthjjhhhhh}
\end{equation}

Now for the above $k$,
by the  H\"{o}lder inequality, we have
\begin{equation}
\begin{array}{rl}
J_1:&=\disp{ \disp\frac{ \chi^2(k-1)}{2C_D}\int_\Omega u ^{{k}+1-m } |\nabla v  |^2}\\
&\leq\disp{ \frac{ \chi^2(k-1)}{2C_D}\left(\disp\int_\Omega u ^{{\frac{N}{N-2}}({k}+1-m )}\right)^{{\frac{N-2}{N}}}\left(\disp\int_\Omega |\nabla v  |^{{N}}\right)^{{\frac{2}{N}}}}\\
&=\disp{ \frac{ \chi^2(k-1)}{2C_D}\|   u ^{\frac{{k}+m-1}{2}}\|^{\frac{2({k}+1-m )}{{k}+m-1}}_{L^{\frac{{\frac{2N}{N-2}}({k}+1-m )}{{k}+m-1}}(\Omega)}
 \|\nabla  v   \|_{L^{{N}}(\Omega)}^2.}\\
\end{array}
\label{cz2.57151hhkkhhhjukildrfthjjhhhhh}
\end{equation}
Next, by \dref{zjscz2.529df763xddfffy}, we derive that $m\geq 1 $, so that, in view of
$N\geq3$,
we have
$$\frac{1}{{k}+m-1}\leq\frac{{k}+1-m}{{k}+m-1}\frac{{\frac{N}{2}}}{({\frac{N}{2}}-1)_+}\leq\frac{N}{(N-2)_+},$$
so that,  by using Gagliardo-Nirenberg interpolation inequality (Lemma \ref{lemma41}) and $L^1(\Omega)$ boundedness for $u $ (see Lemma \dref{wsdelemma45})
we
have
\begin{equation}
\begin{array}{rl}
&\disp{\frac{ \chi^2(k-1)}{2C_D}\|   u ^{\frac{{k}+m-1}{2}}\|
^{\frac{2({k}+1-m )}{{k}+m-1}}_{L^{\frac{{\frac{2N}{N-2}}({k}+1-m )}{{k}+m-1}}(\Omega)}}
\\
\leq&\disp{C_1(\|\nabla    u ^{\frac{{k}+m-1}{2}}\|_{L^2(\Omega)}^{\mu_1}\|  u ^{\frac{{k}+m-1}{2}}\|_{L^\frac{2}{{k}+m-1}(\Omega)}^{1-\mu_1}+\|  u ^{\frac{{k}+m-1}{2}}\|_{L^\frac{2}{{k}+m-1}(\Omega)})^{\frac{2({k}+1-m )}{{k}+m-1}}}\\
\leq&\disp{C_2(\|\nabla    u ^{\frac{{k}+m-1}{2}}\|_{L^2(\Omega)}^{\frac{2({k}+1-m )\mu_1}{{k}+m-1}}+1)}\\
\end{array}
\label{czzsxcv2.563022222ikopl2sdfg44}
\end{equation}
with some positive constants $C_{1}$ as well as $C_2$ and
$$\mu_1=\frac{\frac{N[{k}+m-1]}{2}-\frac{N({k}+m-1)}{{\frac{2N}{N-2}}({k}+1-m )}}{1-\frac{N}{2}+\frac{N[{k}+m-1]}{2}}=
[{k}+m-1]\frac{\frac{N}{2}-\frac{N}{{\frac{2N}{N-2}}({k}+1-m )}}{1-\frac{N}{2}+\frac{N[{k}+m-1]}{2}}\in(0,1).$$

On the other hand, due to Lemma \ref{lemma41} and the fact that $\beta\geq\bar{\beta}>\frac{N}{2}$ and $N>2$, we have
\begin{equation}
\begin{array}{rl}
 \|\nabla  v   \|_{L^{{N}}(\Omega)}^2=&\disp{\| |\nabla v  |^\beta\|_{L^{\frac{{N}}{\beta}}(\Omega)}^{\frac{2}{\beta}}}
\\
\leq&\disp{C_3(\|\nabla |\nabla v  |^\beta\|_{L^2(\Omega)}^{\frac{2\mu_2}{\beta}}\| |\nabla v  |^\beta\|_{L^\frac{2q_{0,***}}{\beta}(\Omega)}^{\frac{2(1-\mu_2)}{\beta}}+\|   |\nabla v  |^\beta\|_{L^\frac{2q_{0,***}}{\beta}
(\Omega)}^{\frac{2}{\beta}})}\\
\leq&\disp{C_{4}(\|\nabla |\nabla v  |^\beta\|_{L^2(\Omega)}^{\frac{2\mu_2}{\beta}}+1),}\\
\end{array}
\label{czcfv2.563022222ikopl255}
\end{equation}
where 
some positive constants $C_3$ as well as  $C_{4}$ and
$$\mu_2=\frac{\frac{N\beta}{2q_{0,***}}-\frac{N\beta}{{N}}}{1-\frac{N}{2}+\frac{N\beta}{2q_{0,***}}}=
\beta\frac{\frac{N}{2q_{0,***}}-\frac{N}{{N}}}{1-\frac{N}{2}+\frac{N\beta}{2q_{0,***}}}\in(0,1).$$
Inserting \dref{czzsxcv2.563022222ikopl2sdfg44}--\dref{czcfv2.563022222ikopl255} into \dref{cz2.57151hhkkhhhjukildrfthjjhhhhh} as well as    by means of the Young inequality and \dref{dcfvzsxcvgcz2.57151hhkkhhhjukildrfthjjhhhhh} we see that for any $\delta>0,$
%
\begin{equation}
\begin{array}{rl}
J_1\leq&\disp{C_{5}(\|\nabla    u ^{\frac{{k}+m-1}{2}}\|_{L^2(\Omega)}^{\frac{2({k}+1-m )\mu_1}{{k}+m-1}}+1)(\|\nabla |\nabla v  |^\beta\|_{L^2(\Omega)}^{\frac{2\mu_2}{\beta}}+1)}\\
=&\disp{C_{5}(\|\nabla    u ^{\frac{{k}+m-1}{2}}\|_{L^2(\Omega)}^{\frac{{N}({{k}+1-m -{\frac{N-2}{N}}})}{1-\frac{N}{2}+\frac{N[{k}+m-1]}{2}}}+1)
(\|\nabla |\nabla v  |^\beta\|_{L^2(\Omega)}^{\frac{{N}(\frac{2}{2q_{0,***}}-\frac{1}{{\frac{N}{2}}})}{1-\frac{N}{2}+\frac{N\beta}{2q_{0,***}}}}+1)}\\
\leq&\disp{\delta\disp\int_{\Omega} |\nabla  u ^{\frac{m+k-1}{2}}|^2+\delta\|\nabla |\nabla v  |^\beta\|_{L^2(\Omega)}^2
+C_{6}~~\mbox{for all}~~ t\in(0,T_{max}) }\\
\end{array}
\label{cz2.57151hhkkhhhjukildrftdfrtgyhu}
\end{equation}
for all
$\beta\geq\bar{\beta}$.
Here we have use the fact that
$k < \frac{2N(m -1)q_{0,***}}{N-2q_{0,***}}(\frac{2}{N}-1+\frac{\beta}{q_{0,***}})+2-m-\frac{2}{N}$
 and $k>2-\frac{2}{N}-m.$
Next, due to  the  H\"{o}lder inequality and $\beta\geq\bar{\beta}>16$, we have
\begin{equation}
\begin{array}{rl}
J_2:&=\disp{ \disp  \int_\Omega u  ^2 |\nabla v  |^{2\beta-2}}\\
&\leq\disp{  \left(\disp\int_\Omega  u  ^{{\frac{\beta}{4}}}\right)^{{\frac{8}{\beta}}}\left(\disp\int_\Omega |\nabla v  |^{{\frac{\beta(2\beta-2)}{\beta-8}}}\right)^{{\frac{\beta-8}{\beta}}}}\\
&\leq\disp{  \left(\disp\int_\Omega  u ^{{\frac{\beta}{4}}}\right)^{{\frac{8}{\beta}}}\left(\disp\int_\Omega |\nabla v  |^{{\frac{\beta(2\beta-2)}{\beta-8}}}\right)^{{\frac{\beta-8}{\beta}}}}\\
&=\disp{  \|   u ^{\frac{{k}+m-1}{2}}\|^{\frac{4}{{k}+m-1}}_{L^{\frac{{\frac{\beta}{2}}}{{k}+m-1}}(\Omega)}
 \|\nabla  v   \|_{L^{{\frac{\beta(2\beta-2)}{\beta-8}}}(\Omega)}^{(2\beta-2)}.}\\
\end{array}
\label{czasxc2.57151hhkkhhhjukildrfttyuijkoghyu66}
\end{equation}

On the other hand, with the help of $k>1-m+{\frac{N-2}{4N}}\beta$, $\beta\geq\bar{\beta}\geq4$ and Lemma  \ref{lemma41} we conclude that
\begin{equation}
\begin{array}{rl}
&\disp{ \|   u ^{\frac{{k}+m-1}{2}}\|^{\frac{4}{{k}+m-1}}_{L^{\frac{{\frac{\beta}{2}}}{{k}+m-1}}(\Omega)}}
\\
\leq&\disp{C_{7}(\|\nabla    u ^{\frac{{k}+m-1}{2}}\|_{L^2(\Omega)}^{\mu_3}\|  u ^{\frac{{k}+m-1}{2}}\|_{L^\frac{2}{{k}+m-1}(\Omega)}^{(1-\mu_3)}+\|   u ^{\frac{{k}+m-1}{2}}\|_{L^\frac{2}{{k}+m-1}(\Omega)})^{\frac{4}{{k}+m-1}}}\\
\leq&\disp{C_{8}(\|\nabla    u ^{\frac{{k}+m-1}{2}}\|_{L^2(\Omega)}^{\frac{4\mu_3}{{k}+m-1}}+1)}\\
\end{array}
\label{czwsdc2.563022222ikopl2sdfgggjjkkk66}
\end{equation}
with some positive constants $C_{7}$ as well as $C_{8}$ and
$$\mu_3=\frac{\frac{N[{k}+m-1]}{2}-\frac{N({k}+m-1)}{{\frac{\beta}{2}}}}{1-\frac{N}{2}+\frac{N[{k}+m-1]}{2}}=
[{k}+m-1]\frac{\frac{N}{2}-\frac{N}{{\frac{\beta}{2}}}}{1-\frac{N}{2}+\frac{N[{k}+m-1]}{2}}\in(0,1).$$

On the other hand, again, it infers by the Gagliardo-Nirenberg inequality (Lemma \ref{lemma41}) that there are $C_{9}>0$  and $C_{10}>0$
ensuring
%
\begin{equation}
\begin{array}{rl}
 \|\nabla  v   \|_{L^{{\frac{\beta(2\beta-2)}{\beta-8}}}(\Omega)}^{(2\beta-2)}=&\disp{\| |\nabla v  |^\beta\|_{L^{\frac{{\frac{\beta(2\beta-2)}{\beta-8}}}{\beta}}(\Omega)}^{\frac{2\beta-2}{\beta}}}
\\
\leq&\disp{C_{9}(\|\nabla |\nabla v  |^\beta\|_{L^2(\Omega)}^{\frac{(2\beta-2)\mu_4}{\beta}}\| |\nabla v  |^\beta\|_{L^\frac{2q_{0,***}}{\beta}(\Omega)}^{\frac{(2\beta-2)(1-\mu_4)}{\beta}}+\|   |\nabla v  |^\beta\|_{L^\frac{2q_{0,***}}{\beta}
(\Omega)}^{\frac{(2\beta-2)}{\beta}})}\\
\leq&\disp{C_{10}(\|\nabla |\nabla v  |^\beta\|_{L^2(\Omega)}^{\frac{(2\beta-2)\mu_4}{\beta}}+1)}\\
\end{array}
\label{czqscvb2.563022222ikopl2gg66}
\end{equation}
for any $\beta\geq\bar{\beta}>\frac{7N+2}{2}$,
where 
$$\mu_4=\frac{\frac{N\beta}{2q_{0,***}}-\frac{N\beta}{{\frac{\beta(2\beta-2)}{\beta-8}}}}{1-\frac{N}{2}+\frac{N\beta}{2q_{0,***}}}=
\beta\frac{\frac{N}{2q_{0,***}}-\frac{N}{{\frac{\beta(2\beta-2)}{\beta-8}}}}{1-\frac{N}{2}+\frac{N\beta}{2q_{0,***}}}\in(0,1).$$
Inserting \dref{czwsdc2.563022222ikopl2sdfgggjjkkk66}--\dref{czqscvb2.563022222ikopl2gg66} into \dref{czasxc2.57151hhkkhhhjukildrfttyuijkoghyu66} and using $k > \frac{N(1-\frac{4}{\beta})}{1+\frac{N}{2q_{0,***}}-\frac{4N}{\beta}}(\frac{2}{N}-1+\frac{\beta}{q_{0,***}})+2-m-\frac{2}{N}$
 and $k>2-\frac{2}{N}-m$
and Lemma \ref{lemma41}, we derive that for the above  $\delta>0,$
\begin{equation}
\begin{array}{rl}
J_2\leq&\disp{C_{11}(\|\nabla    u ^{\frac{{k}+m-1}{2}}\|_{L^2(\Omega)}^{
\frac{{N}(2-{\frac{8}{\beta}})}{1-\frac{N}{2}+\frac{N[{k}+m-1]}{2}}}+1)
(\|\nabla |\nabla v  |^\beta\|_{L^2(\Omega)}^{\frac{{N}(\frac{2\beta-2}{2q_{0,***}}-{\frac{\beta-8}{\beta}})}{1-\frac{N}{2}+\frac{N\beta}{2q_{0,***}}}}+1)}\\
\leq&\disp{\delta\disp\int_{\Omega} |\nabla  u ^{\frac{m+k-1}{2}}|^2+\delta\|\nabla |\nabla v  |^\beta\|_{L^2(\Omega)}^2
+C_{12}~~\mbox{for all}~~ t\in(0,T_{max}).  }\\
\end{array}
\label{cz2.57151hhkkhhhjukildrftdfrtgyhugg}
\end{equation}
Finally, with the help of  \dref{cz2.5ghju48cfg924ghyuji} and by the Sobolev inequality, the Young inequality and \dref{dcfvzsxcvgcz2.57151hhkkhhhjukildrfthjjhhhhh}, we conclude that for the above  $\delta>0,$ there exist positive constants $C_{13},C_{14}$ as well as $C_{15}$
and $C_{16}$ such that
\begin{equation}\label{hhcfvgbhhjui909klopji115}
\begin{array}{rl}
J_3=&\disp{\int_\Omega v ^{k+1}}\\
\leq &\disp{C_{13}\|v \|^{k+1}_{L^\infty(\Omega)}}\\
\leq &\disp{C_{14}(\|\nabla v \|^{k+1}_{L^{N+1}(\Omega)}+1)}\\
\leq &\disp{C_{15}(\|\nabla v \|^{k+1}_{L^{2\beta}(\Omega)}+1)}\\
\leq &\disp{\delta\|\nabla v \|^{2\beta}_{L^{2\beta}(\Omega)}+C_{16}}\\
\end{array}
\end{equation}
for
all $\beta\geq\bar{\beta}>N+1.$
Now, inserting \dref{cz2.57151hhkkhhhjukildrftdfrtgyhu},  \dref{cz2.57151hhkkhhhjukildrftdfrtgyhugg}--\dref{hhcfvgbhhjui909klopji115} into \dref{1234hjui909klopji115} and using the Young inequality and choosing $\delta$ small enough yields to
\begin{equation}\label{dfvvfdcvfbhjui909klopji115}
\begin{array}{rl}
&\disp{\frac{d}{dt}(\frac{1}{k}\| u  \|^{k}_{L^k(\Omega)}+\frac{1}{{2\beta}}\|\nabla v \|^{{{2\beta}}}_{L^{{2\beta}}(\Omega)})+\frac{3(\beta-1)}{8{\beta^2}}\disp\int_{\Omega}\left|\nabla |\nabla v |^{\beta}\right|^2+\frac{\mu}{2}\int_{\Omega} u  ^{k+1}}\\
&+\disp{\frac{1}{2}\disp\int_\Omega  |\nabla v |^{2\beta-2}|D^2v  |^2+\disp\frac{1}{2}\int_{\Omega} |\nabla v |^{2\beta}+\frac{(k-1) C_D }{8}\int_{\Omega} u  ^{m+k-3}|\nabla u  |^2{}}\\
\leq&\disp{C_{16}~~\mbox{for all}~~ t\in(0,T_{max})}\\
\end{array}
\end{equation}
and  some positive constant $C_{16}$.
Therefore, letting
  $y:=\disp\int_{\Omega} u  ^{k}  +\disp\int_{\Omega} |\nabla  v   |^{2\beta} $ 
in \dref{dfvvfdcvfbhjui909klopji115} yields to
$$\frac{d}{dt}y(t)+C_{17}y(t)\leq C_{18}~~\mbox{for all}~~ t\in(0,T_{max}).$$
 Thus a standard ODE comparison argument implies
boundedness of $y(t)$ for all $t\in (0, T_{max})$. Clearly, $\| u  (\cdot, t)\|_{L^k(\Omega)}$ and $\|\nabla  v   (\cdot, t)\|_{L^{2\beta}(\Omega)}$
are bounded for all $t\in (0, T_{max})$.
Obviously,
$\lim_{\beta\rightarrow+\infty}\disp{ \frac{N(1-\frac{4}{\beta})}{1+\frac{N}{2q_{0,***}}-\frac{4N}{\beta}}(\frac{2}{N}-1+\frac{\beta}{q_{0,***}})}=
\lim_{\beta\rightarrow+\infty}\frac{2N(m -1)q_{0,***}}{N-2q_{0,***}}(\frac{2}{N}-1+\frac{\beta}{q_{0,***}})+2-m-\frac{2}{N}=+\infty,$
hence,
the boundedness of $\| u  (\cdot, t)\|_{L^k(\Omega)}$ and the H\"{o}lder inequality implies the results.
\end{proof}

Employing Lemmas \ref{lemma456ssddfff30}--\ref{lemmaddff45ssdd630}, we can prove the following lemma.

\begin{lemma}\label{sddlemmddffa45630}
Let $\Omega\subset \mathbb{R}^N(N\geq1)$ be a bounded domain with smooth boundary. Furthermore,
assume that $m>\frac{2N}{N+\gamma_*}$ with $N\geq1$,
 where $\gamma_*$ is given by \dref{xeerr1.731426677gg}. Then
 for any $k>1$, there  exists a positive constant  $C$  such that
\begin{equation}
\|u (\cdot, t)\|_{L^k(\Omega)}\leq C ~~\mbox{for all}~~ t\in(0, T_{max}).
\label{zjscz2.529ddff7xssdd9630xxy}
\end{equation}
\end{lemma}
Along with the Duhamel's principle and $L^p$-$L^q$ estimates for Neumann heat semigroup, the above lemma yields the following Lemma.


\begin{lemma}\label{lemdffmddffa45630}
Let $(u , v ,w )$ be the solution of the problem  \dref{7101.2x19x3sss189}. Then
\begin{equation}
\|v (\cdot, t)\|_{W^{1,\infty}(\Omega)}\leq C ~~\mbox{for all}~~ t\in(0, T_{max})
\label{zjscz2.ssdd529ddff7xssdd9630xxy}
\end{equation}
\end{lemma}
\begin{proof}
By Duhamel's principle, we see that the solution $v $ can be expressed as follows
\begin{equation}
v (t)=e^{t(\Delta-1)}v_0 +\int_{0}^{t}e^{(t-s)(\Delta-1)}v (s) ds,~ t\in(0, T_{max}),
\label{5555hhjjjfghbnmcz2.5ghju48cfg924ghyuji}
\end{equation}
where $(e^{ t \Delta})_{t\geq0}$ is the Neumann heat semigroup in $\Omega$.  Using Lemma \ref{sddlemmddffa45630}, we
follow the $L^p$-$L^q$ estimates for Neumann heat semigroup to obtain
for any $t\in(0,T_{max})$, we obtain
\begin{equation}
\begin{array}{rl}
&\| v  (\cdot, t)\|_{W^{1,\infty}(\Omega)}\\
\leq&\disp{e^{- t}\| \nabla v_0\|_{L^{\infty}(\Omega)}+\int_{0}^t(t-s)^{-\frac{N}{2}(\frac{1}{2N}-\frac{1}{2})}e^{-(t-s)}
\|u (\cdot,s)\|_{L^{2N}(\Omega)}ds}\\
\leq&\disp{C_{1}~~ \mbox{for all}~~ t\in(0,T_{max}).}\\
\end{array}
\label{zjccffgbhjcvvvbscz2.5297x96301ku}
\end{equation}
 This lemma is proved.
\end{proof}


Applying Lemma \ref{sddlemmddffa45630} and Lemma  \ref{lemdffmddffa45630}, a straightforward adaptation of the well-established Moser-type iteration procedure \cite{Alikakos72}
allows us to formulate a general condition which is sufficient for the boundedness of $u $.

\begin{lemma}\label{lemdffssddmddffa45630}
Let $(u , v ,w )$ be the solution of the problem  \dref{7101.2x19x3sss189}. Then there exists a positive constant $C$ such that
\begin{equation}
\|u (\cdot, t)\|_{L^{\infty}(\Omega)}\leq C ~~\mbox{for all}~~ t\in(0, T_{max}).
\label{zjscz2.ssdd529ddff7xssddsddd9630xxy}
\end{equation}
\end{lemma}
\begin{proof}
Firstly, by Lemma \ref{sddlemmddffa45630}, we obtain that for any $k > 0$,
\begin{equation}
\|u (\cdot,t)\|_{L^k(\Omega)}\leq \alpha_k~~\mbox{for all}~~ t\in(0, T_{max}),
\label{zjscz2.5297x9630111rrddtt}
\end{equation}
where $\alpha_k$ depends on $k$.
Multiplying the first equation of  \dref{7101.2x19x3sss189} by $ku^{k-1} $
with $k\geq\max\{2m, N+2\}$, integrating it over $\Omega$, then using the boundary condition
$\frac{\partial u }{\partial \nu}=0$ and  combining with Lemma \ref{lemdffmddffa45630}, we obtain
\begin{equation}
\begin{array}{rl}
&\disp{\frac{d}{dt}\|u \|^{k}_{L^k(\Omega)}+k (k-1)  C_D \int_{\Omega} (u +\varepsilon)^{m-1}u ^{k-2}|\nabla u  |^2{}+
\mu k\int_{\Omega} u ^{k+1}+\int_{\Omega} u ^k{}}
\\
\leq&\disp{k(k-1)\chi\int_{\Omega} u ^{k-1}\nabla u \cdot\nabla v -\xi\int_\Omega \nabla\cdot(  u  \nabla w )
   u ^{k-1}+(\mu k+1)\int_{\Omega} u ^k}\\
   \leq&\disp{k(k-1)\chi\int_{\Omega} u ^{k-1}\nabla u \cdot\nabla v }\\
   &\disp{+\frac{({k-1})}{k}\xi\|w_0\|_{L^\infty(\Omega)}\int_\Omega u^k  v  dx+\kappa \frac{({k-1})}{k}\xi \int_\Omega u^k  dx+(\mu k+1)\int_{\Omega} u ^k}\\
\leq&\disp{\frac{k (k-1)  C_D }{2}\int_{\Omega} (u +\varepsilon)^{m-1}u ^{k-2}|\nabla u  |^2+C_1k\int_{\Omega} u ^{k}+C_2k^2\int_{\Omega} u ^k(u +\varepsilon)^{1-m},}\\
\end{array}
\label{vgbhnsxcdvfddfcz2.5xxdsddddddf1jjssdsddffjj}
\end{equation}
by using \dref{cz2.563019rrtttt12},
where $C_{1}>0,C_{2}>0$, as all subsequently appearing constants $C_i(i= 3, 4, \ldots)$ are independent of $k$.
In what follows, we estimate the last two terms of \dref{vgbhnsxcdvfddfcz2.5xxdsddddddf1jjssdsddffjj}.
When $m\geq 1$, then, $$(u +\varepsilon)^{1-m}\leq u ^{1-m}~~~\mbox{and}~~~(u +\varepsilon)^{m-1}\geq u ^{m-1},$$
so that,
 by \dref{vgbhnsxcdvfddfcz2.5xxdsddddddf1jjssdsddffjj}, we have
 \begin{equation}
\begin{array}{rl}
&\disp{\frac{d}{dt}\|u \|^{k}_{L^k(\Omega)}+\frac{k (k-1)  C_D }{2}\int_{\Omega} u ^{m+k-3}|\nabla u  |^2{}+\int_{\Omega} u ^k{}}
\\
\leq&\disp{C_1k\int_{\Omega} u ^{k}+C_2k^2\int_{\Omega} u ^{k+1-m}.}\\
\end{array}
\label{vgbhnsxcdvfddfcz2.5xxdsddddddf1jjssdssddsddffjj}
\end{equation}
In order to take full advantage of the dissipated quantities appearing on the left-hand side herein,  for any $\delta>0,$
we
first invoke the Gagliardo-Nirenberg inequality which provides $C_3>0,C_4 > 0$ as well as $C_5 > 0$ and $C_6 > 0$ such that
\begin{equation}
\begin{array}{rl}
&\disp{C_1k \|   u \|^{k}_{L^{k}(\Omega)}}\\
=&\disp{ C_1k \|   u ^{\frac{{k}+m-1}{2}}\|^{\frac{2k}{{k}+m-1}}_{L^{\frac{{2k}}{{k}+m-1}}(\Omega)}}
\\
\leq&\disp{C_{1}k\|\nabla    u ^{\frac{{k}+m-1}{2}}\|_{L^2(\Omega)}^{\frac{2Nk}{(N+2){k}+2N(m-1)}}\|  u ^{\frac{{k}+m-1}{2}}\|_{L^\frac{k}{{k}+m-1}(\Omega)}^{\frac{2k}{{k}+m-1}-\frac{2Nk}{(N+2){k}+2N(m-1)}}+C_{3}k\|   u ^{\frac{{k}+m-1}{2}}\|_{L^\frac{k}{{k}+m-1}(\Omega)}^{\frac{2k}{{k}+m-1}}}\\
\leq&\disp{\delta\|\nabla    u ^{\frac{{k}+m-1}{2}}\|_{L^2(\Omega)}^{2}+C_4 k^{\frac{(N+2)k+2N(m-1)}{2k+2N(m-1)}}\|  u ^{\frac{{k}+m-1}{2}}\|_{L^\frac{k}{{k}+m-1}(\Omega)}^{\frac{k[2k+N(m-1)]}{({k}+m-1)(k+N(m-1))}}+C_{3}k\|   u ^{\frac{{k}+m-1}{2}}\|_{L^\frac{k}{{k}+m-1}(\Omega)}^{\frac{2k}{{k}+m-1}}}\\
\leq&\disp{\delta\|\nabla    u ^{\frac{{k}+m-1}{2}}\|_{L^2(\Omega)}^{2}+C_4 k^{\frac{N+2}{2}} \|  u \|_{L^{\frac{k}{2}}(\Omega)}^{\frac{k[2k+N(m-1)]}{(2k+2N(m-1))}}+C_{3}k\|   u \|_{L^\frac{k}{2}(\Omega)}^{k}}\\
\end{array}
\label{czwsdc2.563022222ikoplssd2sdfgggjjkkk66}
\end{equation}
and
\begin{equation}
\begin{array}{rl}
&\disp{C_2k^2 \|   u \|^{k+1-m}_{L^{k+1-m}(\Omega)}}\\
=&\disp{ C_2k^2 \|   u ^{\frac{{k}+m-1}{2}}\|^{\frac{2(k+1-m)}{{k}+m-1}}_{L^{\frac{{2(k+1-m)}}{{k}+m-1}}(\Omega)}}
\\
\leq&\disp{C_{5}k^2\|\nabla    u ^{\frac{{k}+m-1}{2}}\|_{L^2(\Omega)}^{\frac{2N[k-(2m-2)]}{(N+2){k}+2N(m-1)}}\|  u ^{\frac{{k}+m-1}{2}}\|_{L^\frac{k}{{k}+m-1}(\Omega)}^{\frac{2(k+1-m)}{{k}+m-1}-\frac{2N[k-(2m-2)]}{(N+2){k}+2N(m-1)}}+C_{5}k^2\|   u ^{\frac{{k}+m-1}{2}}\|_{L^\frac{k}{{k}+m-1}(\Omega)}^{\frac{2(k-1+m)}{{k}+m-1}}}\\
\leq&\disp{\delta\|\nabla    u ^{\frac{{k}+m-1}{2}}\|_{L^2(\Omega)}^{2}+C_6 k^{\frac{(N+2)k+2N(m-1)}{k+2N(m-1)}}\|  u ^{\frac{{k}+m-1}{2}}\|_{L^\frac{k}{{k}+m-1}(\Omega)}^{\frac{k[k+2(m-1)]}{({k}+m-1)(\frac{k}{2}+N(m-1))}}+C_{5}k^2\|   u ^{\frac{{k}+m-1}{2}}\|_{L^\frac{k}{{k}+m-1}(\Omega)}^{\frac{2(k-1+m)}{{k}+m-1}}}\\
\leq&\disp{\delta\|\nabla    u ^{\frac{{k}+m-1}{2}}\|_{L^2(\Omega)}^{2}+C_6 k^{N+2}\|  u \|_{L^{\frac{k}{2}}(\Omega)}^{\frac{k[k+2(m-1)]}{(k+2N(m-1))}}+C_{5}k^2\|   u \|_{L^\frac{k}{2}(\Omega)}^{k+1-m}.}\\
\end{array}
\label{czwsdc2.563022222ikoplddffssd2sdfgggjjkkk66}
\end{equation}
Taking $\delta$ appropriately small, and substituting the above two inequalities into \dref{vgbhnsxcdvfddfcz2.5xxdsddddddf1jjssdssddsddffjj}, we obtain
\begin{equation}
\begin{array}{rl}
&\disp{\frac{d}{dt}\|u \|^{k}_{L^k(\Omega)}+\int_{\Omega} u ^k{}}
\\
\leq&\disp{C_7 k^{\frac{N+2}{2}} \|  u \|_{L^{\frac{k}{2}}(\Omega)}^{\frac{k[2k+N(m-1)]}{(2k+2N(m-1))}}+C_8 k^{N+2}\|  u \|_{L^{\frac{k}{2}}(\Omega)}^{\frac{k[k+2(m-1)]}{(k+2N(m-1))}}+C_{9}k\|   u \|_{L^\frac{k}{2}(\Omega)}^{k}+C_{10}k^2\|   u \|_{L^\frac{k}{2}(\Omega)}^{k+1-m}.}\\
\end{array}
\label{vgbhnsxcdvfddfcssddz2.5xxdsddddddf1jjssdssddsddffjj}
\end{equation}
Notice that
$${\frac{k[k+2(m-1)]}{(k+2N(m-1))}}\leq\max\{\frac{k[2k+N(m-1)]}{(2k+2N(m-1))},k+1-m\}\leq k,$$
 we further obtain
 \begin{equation}
\begin{array}{rl}
\disp{\frac{d}{dt}\|u \|^{k}_{L^k(\Omega)}+\int_{\Omega} u ^k{}}
\leq&\disp{C_{11} k^{N+2}\|  u \|_{L^{\frac{k}{2}}(\Omega)}^{\frac{k[k+2(m-1)]}{(k+2N(m-1))}}+C_{12}k^{N+2}\|   u \|_{L^\frac{k}{2}(\Omega)}^{k}.}\\
\end{array}
\label{vgbhnsxcdvfddfcsssddsddz2.5xxdsddddddf1jjssdssddsddffjj}
\end{equation}
In what follows, we use Moser iteration method to show the $L^\infty$ estimate of $u $. Take $k_i=
2k_{i-1} = 2^ik_0 , k_0 = 2m$, $M_i =\sup_{t\in(0,T_{max})}\int_{\Omega}u ^{\frac{{k_i}}{2}}$, then when $m\geq 1$, we have
\begin{equation}
\begin{array}{rl}
\disp{M_i}\leq&\disp{\max\{\lambda^iM_{i-1}^2,\|1+u_0\|_{L^\infty(\Omega)}^{k_i}\}.}\\
\end{array}
\label{vgbhnsxcdvfddfcsssddsddz2.5xxdsdddddssdddf1jjssdssddsddffjj}
\end{equation}
with some $\lambda> 1.$
Now if $\lambda^iM_{i-1}^2\leq\|1+u_0\|_{L^\infty(\Omega)}^{k_i}$
for infinitely many $i\geq 1$, we get (3.33) with $C = \|1+u_0\|_{L^\infty(\Omega)}$.
Otherwise $M _i\leq \lambda^iM_{i-1}^2$
for all $i = 0, 1, \ldots$, so $\ln M_i \leq i \ln \lambda +2 \ln M_{i-1}$. By induction, we get
$$\ln M_i \leq (i+2)\ln \lambda +2^i(\ln M_{0}+2\ln \lambda)$$
and thus
$$M_i\leq \lambda^{i+2+2^i}M_0^{2^i}. $$
From this, it follows that
\dref{zjscz2.ssdd529ddff7xssddsddd9630xxy}  is valid with some positive constant.
When $0< m<1$, noticing that $u^k  (u +\varepsilon)^{1-m}\leq u ^{k+1-m}
+\varepsilon^{1-m} u ^{k}$, then by \dref{zjscz2.5297x9630111rrddtt} and $\varepsilon\in(0,1)$, we derive from Lemma \ref{lemdffmddffa45630} that
 \begin{equation}
\begin{array}{rl}
&\disp{\frac{d}{dt}\|u \|^{k}_{L^k(\Omega)}+\frac{k (k-1)  C_D }{2^{2-m}}\int_{u >\varepsilon} u ^{m+k-3}|\nabla u  |^2{}+\mu k\int_{\Omega} u ^{k+1}+\mu k\int_{\Omega} u ^{k}{}}
\\
\leq&\disp{C_{13}k^2\int_{\Omega} u ^{k}+C_{14}k^2\int_{\Omega} u ^{k+1-m}}\\
\leq&\disp{C_{13}k^2\int_{u \leq\varepsilon} u ^{k}+C_{13}k^2\int_{u >\varepsilon} u ^{k}+C_{14}k^2\int_{u \leq \varepsilon} u ^{k+1-m}+C_{14}k^2\int_{u >\varepsilon} u ^{k+1-m}}\\
\leq&\disp{C_{13}k^2\int_{u \leq \varepsilon} u ^{\frac{k}{2}}+C_{14}k^2\int_{u \geq\varepsilon} u ^{k}+C_{14}k^2\int_{u \geq \varepsilon} u ^{k+1-m}.}\\
\end{array}
\label{vgbhnsxcdvfddfcz2ssdd.5xxdsddddddf1jjssdssddsddffjj}
\end{equation}

Denote $J(t)=\{x\in  u >\varepsilon\}$. By virtue of the Gagliardo-Nirenberg interpolation inequality and \dref{zjscz2.5297x9630111rrddtt}, we obtain
$$\|u \|^k_{L^k(J(t))}\leq\|u \|_{L^3(J(t))}\|u \|^{k-1}_{L^{\frac{3(k-1)}{2}}(J(t))}\leq
C_{15}
\|u \|^{k-1}_{L^{\frac{3(k-1)}{2}}(J(t))}$$
and
$$\|u \|^{k+1-m}_{L^{k+1-m}(J(t))}\leq\|u \|_{L^{3(2-m)}(J(t))}^{2-m}\|u \|^{k-1}_{L^{\frac{3(k-1)}{2}}(J(t))}\leq
C_{16}
\|u \|^{k-1}_{L^{\frac{3(k-1)}{2}}(J(t))}.$$
Substituting the above two inequalities into \dref{vgbhnsxcdvfddfcz2.5xxdsddddddf1jjssdsddffjj}, we obtain
\begin{equation}
\begin{array}{rl}
&\disp{\frac{d}{dt}\|u \|^{k}_{L^k(\Omega)}+k (k-1)  C_D \int_{J(t)} (u +\varepsilon)^{m-1}u ^{k-2}|\nabla u  |^2{}+
\mu k\int_{\Omega} u ^{k+1}+\int_{\Omega} u ^k{}}
\\
\leq&\disp{C_{17}k^2\int_{\Omega} u ^{\frac{k}{2}}+C_{18}k^2\|u \|^{k-1}_{L^{\frac{3(k-1)}{2}}(J(t))}.}\\
\end{array}
\label{vgbhnsxcdvfddfczssdd2.5xxdsddddddf1jjssdsddffjj}
\end{equation}

Using the Gagliardo-Nirenberg interpolation inequality again, it follows
$$\begin{array}{rl}
&C_{18}k^2\|u \|^{k-1}_{L^{\frac{3(k-1)}{2}}(J(t))}\\
=&C_{18}k^2\|u ^{\frac{{k}+m-1}{2}}\|^{\frac{2(k-1)}{m+k-1}}_{L^{\frac{3(k-1)}{k+m-1}}(J(t))}\\
\leq&\disp{C_{19}k^2\|\nabla    u ^{\frac{{k}+m-1}{2}}\|_{L^2(J(t))}^{\frac{4(2k-N)}{(N+2){k}+2N(m-1)}}\|  u ^{\frac{{k}+m-1}{2}}\|_{L^\frac{k}{{k}+m-1}(\Omega)}^{\frac{2(k-1)}{m+k-1}-\frac{4(2k-N)}{(N+2){k}+2N(m-1)}}+C_{19}k^2\|   u ^{\frac{{k}+m-1}{2}}\|_{L^\frac{k}{{k}+m-1}(\Omega)}^{\frac{2(k-1)}{{k}+m-1}}}\\
\leq&\disp{\delta\|\nabla    u ^{\frac{{k}+m-1}{2}}\|_{L^2(J(t))}^{2}+C_{20} k^{\frac{2[(N+2){k}+2N(m-1)]}{(N-2)k+2Nm}}\|  u ^{\frac{{k}+m-1}{2}}\|_{L^\frac{k}{{k}+m-1}(\Omega)}^{\frac{2k[k(N-2)+(2N-4)m+2-N]}{(m+k-1)[(N-2){k}+2Nm]}}+C_{19}k^{2}\|   u \|_{L^\frac{k}{2}(\Omega)}^{k-1}}\\
\leq&\disp{\delta\|\nabla    u ^{\frac{{k}+m-1}{2}}\|_{L^2(J(t))}^{2}+C_{20} k^{\frac{2(N+2)}{(N-2)}}\|  u ^{\frac{{k}+m-1}{2}}\|_{L^\frac{k}{{k}+m-1}(\Omega)}^{\frac{2k[k(N-2)+(2N-4)m+2-N]}{(m+k-1)[(N-2){k}+2Nm]}}+C_{19}k^{2}\|   u \|_{L^\frac{k}{2}(\Omega)}^{k-1}}\\
=&\disp{\delta\|\nabla    u ^{\frac{{k}+m-1}{2}}\|_{L^2(J(t))}^{2}+C_{20} k^{\frac{2(N+2)}{(N-2)}}\|  u ^{\frac{{k}+m-1}{2}}\|_{L^\frac{k}{2}(\Omega)}^{\frac{k[k(N-2)+(2N-4)m+2-N]}{[(N-2){k}+2Nm]}}+C_{19}k^{2}\|   u \|_{L^\frac{k}{2}(\Omega)}^{k-1},}\\
\end{array}
$$
substituting this inequality into \dref{vgbhnsxcdvfddfczssdd2.5xxdsddddddf1jjssdsddffjj} gives
\begin{equation}
\begin{array}{rl}
&\disp{\frac{d}{dt}\|u \|^{k}_{L^k(\Omega)}+k (k-1)  C_D \int_{J(t)} (u +\varepsilon)^{m-1}u ^{k-2}|\nabla u  |^2{}+
\mu k\int_{\Omega} u ^{k+1}+\int_{\Omega} u ^k{}}
\\
\leq&\disp{C_{21}k^2\int_{\Omega} u ^{\frac{k}{2}}+C_{23} k^{\frac{2(N+2)}{(N-2)}}\|  u ^{\frac{{k}+m-1}{2}}\|_{L^\frac{k}{2}(\Omega)}^{\frac{k[k(N-2)+(2N-4)m+2-N]}{[(N-2){k}+2Nm]}}+C_{22}k^{2}\|   u \|_{L^\frac{k}{2}(\Omega)}^{k-1}.}\\
\end{array}
\label{vgbhnsxcdvfddfczssdd2.5xssdxdsddddddf1jjssdsddffjj}
\end{equation}
Similarly to the case $m\geq 1$, and we obtain \dref{zjscz2.ssdd529ddff7xssddsddd9630xxy}.

\end{proof}

By virtue of \dref{1.163072x} and Lemmas \ref{lemdffmddffa45630}--\ref{lemdffssddmddffa45630}, we are now in a position to prove Theorem \ref{theorem3}.

{\bf Proof of Theorem \ref{theorem3}.}
\begin{proof}
From \dref{1.163072x} and Lemmas \ref{lemdffmddffa45630}--\ref{lemdffssddmddffa45630}, we derive  that there
exists a constant $C > 0$ independent of $\varepsilon$ such that
\begin{equation}
\|u (\cdot,t)\|_{L^\infty(\Omega)} +\|v (\cdot,t)\|_{W^{1,\infty}(\Omega)}  \leq C ~~\mbox{for all}~~ t\in(0,T_{max}).
\label{zjscz2.5297x9630111kkuu}
\end{equation}
Suppose on the contrary that $T_{max}< \infty$, then \dref{zjscz2.5297x9630111kkuu} contradicts to the blow-up criterion \dref{1.163072x}, which
implies $T_{max}= \infty$. Therefore, the classical solution $(u,v,w)$ is global in time and bounded.
\end{proof}



\section{Proof of Theorem \ref{theossdrem3}}

The goal of this section is to prove theorem \ref{theossdrem3}. In the absence of \dref{91ssdd61}, the first equation in system \dref{7101.2x19x3sss189} may be degenerate, so system \dref{7101.2x19x3sss189} might not
have classical solutions. Our goal is to construct solutions of \dref{7101.2x19x3sss189} as limits of solutions to appropriately regularized problems.
To this end,
we approximate the diffusion coefficient function in \dref{7101.2x19x3sss189}
 by a family $(D_\varepsilon)_{\varepsilon\in(0,1)}$ of
functions
$$D_\varepsilon\in C^2((0,\infty))~ \mbox{such that}~~ D_\varepsilon(u)\geq\varepsilon
 ~\mbox{for all}~ u > 0 $$$$~~\mbox{and}~
D(u) \leq D_\varepsilon(u)\leq D(u)  + 2\varepsilon~ \mbox{for all}~
 u > 0 ~\mbox{and}~\varepsilon\in (0, 1).$$
  Therefore, for any $\varepsilon\in(0,1)$, the regularized problem of \dref{7101.2x19x3sss189} is presented as follows
%
\begin{equation}
 \left\{\begin{array}{ll}
  u_{\varepsilon t}=\nabla\cdot(D_\varepsilon(u_\varepsilon)\nabla u_\varepsilon)-\chi\nabla\cdot(u_\varepsilon\nabla v_\varepsilon)-\xi\nabla\cdot
  (u_\varepsilon\nabla w_\varepsilon)+\mu u_\varepsilon(1-u_\varepsilon-w_\varepsilon),\quad
x\in \Omega, t>0,\\
 \disp{v_{\varepsilon t}=\Delta v_\varepsilon +u_\varepsilon- v_\varepsilon},\quad
x\in \Omega, t>0,\\
\disp{w_{\varepsilon t}=- v_\varepsilon w_\varepsilon },\quad
x\in \Omega, t>0,\\
 \disp{\frac{\partial u_\varepsilon}{\partial \nu}=\frac{\partial v_\varepsilon}{\partial \nu}=\frac{\partial w_\varepsilon}{\partial \nu}=0},\quad
x\in \partial\Omega, t>0,\\
\disp{u_\varepsilon(x,0)=u_0(x)},v_\varepsilon(x,0)=v_0(x),w_\varepsilon(x,0)=w_0(x),\quad
x\in \Omega.\\
 \end{array}\right.\label{5555710ssdd1.2x19x3189}
\end{equation}
We are now in the position to construct global weak solutions for \dref{7101.2x19x3sss189}. Before going
into details, let us first give the definition of weak solution.

\begin{definition}\label{df1}
Let  $T\in(0,\infty]$, and  $\Omega\subset R^N$ be a bounded domain with smooth boundary.
A triple $(u, v, w)$ of nonnegative functions defined on  $\Omega\times(0,T)$, is called a weak solution to
\dref{7101.2x19x3sss189}  on $[0, T )$ if
%
%
\begin{equation}
\begin{array}{ll}
   \mbox{(i)}~~ u\in L_{loc}^2([0,T);L^2(\Omega)),~~
    v \in L_{loc}^2([0,T); W^{1,2}(\Omega)),~~
    w \in L_{loc}^2([0,T); W^{1,2}(\Omega));\\
\end{array}\label{dffff1.1fghyuisdakkklll}
\end{equation}
\begin{equation}\label{726291hh}
\begin{array}{rl}
 \mbox{(ii)}~~ H(u) \in L^1_{loc}(\bar{\Omega}\times [0, T)),~~u\nabla v~~\mbox{and}~~u\nabla w~~ \in
L^1_{loc}(\bar{\Omega}\times [0, T);\mathbb{R}^{N});
\end{array}
\end{equation}
(iii) $(u, v, w)$ satisfies \dref{7101.2x19x3sss189}  in the sense that for every  $\varphi\in C_0^{\infty} (\bar{\Omega}\times[0, T))$
\begin{equation}
\begin{array}{rl}\label{eqx45xx12112ccgghh}
&\disp{-\int_0^{T}\int_{\Omega}u\varphi_t-\int_{\Omega}u_0\varphi(\cdot,0) }
\\
=&\disp{
\int_0^T\int_{\Omega}H(u)\Delta\varphi+\chi\int_0^T\int_{\Omega}u
\nabla v\cdot\nabla\varphi+\xi\int_0^T\int_{\Omega}u
\nabla w\cdot\nabla\varphi}
+\disp{\int_0^T\int_{\Omega}(\mu u-uw-\mu u^2)\varphi;}
\end{array}
\end{equation}
holds as well as
\begin{equation}
\begin{array}{rl}\label{eqx45xx12112ccgghhjj}
&\disp{-\int_0^{T}\int_{\Omega}v\varphi_t-\int_{\Omega}v_0\varphi(\cdot,0)=-
\int_0^T\int_{\Omega}\nabla v\cdot\nabla\varphi-\int_0^T\int_{\Omega}v\varphi+\int_0^T\int_{\Omega}u\varphi}
\end{array}
\end{equation}
and
\begin{equation}
\begin{array}{rl}\label{eqx45xx1sddd2112ccgghhjj}
&\disp{-\int_0^{T}\int_{\Omega}w\varphi_t-\int_{\Omega}w_0\varphi(\cdot,0)=-
\int_0^T\int_{\Omega}vw\varphi,}
\end{array}
\end{equation}
where we let
\begin{equation}
H(s)=\int_0^sD(\sigma)d\sigma~~\mbox{for}~~s\geq0.
 \label{zjscz2.5297x963ssddd0111kkuu}
\end{equation}
In particular, if  $T=\infty$ can be taken, then $(u, v, w)$ is called a global-in-time weak solution
to \dref{7101.2x19x3sss189}.
%
%



We proceed to establish the main step towards the boundedness of weak solutions to  \dref{7101.2x19x3sss189}.
To this end, firstly, from \dref{1.163072x} and Lemmas \ref{lemdffmddffa45630}--\ref{lemdffssddmddffa45630},
 we can easily derive the following  estimates for $u_\varepsilon$ and $v_\varepsilon$, which plays an important role in proving Theorem \ref{theossdrem3}.
%
%
\begin{lemma}\label{lemma45630hhuujjuu}
Let $m> \frac{2N}{N+\gamma_*}$ with $N\geq1$,
 where $\gamma_*$ is given by \dref{xeerr1.731426677gg}.  Then one can find 
$C > 0$
independent of $\varepsilon\in(0, 1)$ such that
\begin{equation}
\|u_\varepsilon(\cdot,t)\|_{L^\infty(\Omega)}  \leq C ~~\mbox{for all}~~ t\in(0,\infty)
\label{zjscz2.5297x9630111kkuu}
\end{equation}
and
\begin{equation}
\|v_\varepsilon(\cdot,t)\|_{W^{1,\infty}(\Omega)}  \leq C ~~\mbox{for all}~~ t\in(0,\infty).
\label{zjscz2.5297x9630111kkhhii}
\end{equation}
\end{lemma}

Now with the above boundedness
information at hand, we may invoke standard parabolic regularity to obtain the H\"{o}lder regularity properties.

\begin{lemma}\label{lemma45630hhuujjuuyy}
Let $m> \frac{2N}{N+\gamma_*}$ with $N\geq1$,
 where $\gamma_*$ is given by \dref{xeerr1.731426677gg}.
Then for any $\varepsilon\in(0, 1)$, one can find $\mu\in(0, 1)$ such that for some $C > 0$
%
%
\begin{equation}
\|v_\varepsilon(\cdot,t)\|_{C^{\mu,\frac{\mu}{2}}(\Omega\times[t,t+1])}  \leq C ~~\mbox{for all}~~ t\in(0,\infty),
\label{zjscz2.5297x9630111kkhhiioo}
\end{equation}
and such that for any $\tau> 0$
 there exists $C(\tau) > 0$ fulfilling
\begin{equation}
\|\nabla v_\varepsilon(\cdot,t)\|_{C^{\mu,\frac{\mu}{2}}(\Omega\times[t,t+1])} \leq C ~~\mbox{for all}~~ t\in(\tau,\infty).
\label{zjscz2.5297x9630111kkhhffrreerrpphh}
\end{equation}
\end{lemma}
\begin{proof}
Firstly, by Lemma  \ref{lemma45630hhuujjuu}, we derive that $ g_\varepsilon $ is bounded in $L^{\infty} (\Omega\times(0, \infty))$, where $g_\varepsilon (x, t) := -v_\varepsilon(x,t) +u_\varepsilon(x,t) $ for all $(x,t)\in \Omega\times(0, \infty)$.
Therefore, in view of the standard parabolic regularity theory  to the second equation  of  \dref{5555710ssdd1.2x19x3189}, one has \dref{zjscz2.5297x9630111kkhhiioo} and \dref{zjscz2.5297x9630111kkhhffrreerrpphh} hold.
%
\end{proof}



To achieve the convergence result, 
%
%
 we need to derive some
regularity properties of time derivatives.
%
%
\begin{lemma}\label{lemma45630hhuujjuuyytt}
Let $m> \frac{2N}{N+\gamma_*}$ with $N\geq1$,
 where $\gamma_*$ is given by \dref{xeerr1.731426677gg}.
%
%
%
%
%
Moreover,
let $\varsigma> m$ and $\varsigma\geq 2(m - 1)$. Then
for all $T > 0$ and $\varepsilon\in(0,1)$, there exists $C(T) > 0$ such that
\begin{equation}
\int_0^T\|\partial_tu_\varepsilon^\varsigma(\cdot,t)\|_{(W^{2,q}(\Omega))^*}dt  \leq C(T)
\label{zjscz2.5297x9630111kkhhiioott4}
\end{equation}
and
\begin{equation}
\|w_\varepsilon(\cdot,t)\|_{W^{1,\infty}(\Omega)}  \leq C(T).
\label{zjsczsdddfff2.5297x9630111kkhhii}
\end{equation}
\end{lemma}
\begin{proof}
 Firstly, with the help of  Lemma \ref{lemma45630hhuujjuu}, for all $\varepsilon\in(0,1),$ we can fix a positive constants $C_1$ such that
 \begin{equation}
u_\varepsilon\leq C_1~~\mbox{and}~~|\nabla v_\varepsilon|  \leq C_1 ~~\mbox{in}~~ \Omega\times(0,\infty).
\label{gbhnzjscz2.5297x9630111kkhhiioo}
\end{equation}

Next, we will prove \dref{zjscz2.5297x9630111kkhhiioott4}. To this end,
for any fixed   $\psi\in C_0^{\infty}(\Omega)$, multiplying the first equation by $u^{\varsigma-1}_{\varepsilon}\psi$, we have
  \begin{equation}
\begin{array}{rl}
&\disp\frac{1}{\varsigma}\int_{\Omega}\partial_{t}u^{\varsigma}_{\varepsilon}(\cdot,t)\cdot\psi\\ =&\disp{\int_{\Omega}u^{\varsigma-1}_{\varepsilon}\left[\nabla\cdot(D_\varepsilon(u_\varepsilon)\nabla u_\varepsilon)-\chi\nabla\cdot(u_\varepsilon\nabla v_\varepsilon)-\xi\nabla\cdot(u_\varepsilon\nabla w_\varepsilon)+\mu u_\varepsilon(1-w_\varepsilon-u_\varepsilon)\right]\cdot\psi }
\\
=&\disp{-(\varsigma-1)\int_\Omega u_{\varepsilon}^{\varsigma-2}D_\varepsilon(u_\varepsilon)|{\nabla} {n}_{\varepsilon}|^2\psi-\int_\Omega u_{\varepsilon}^{\varsigma-1}D_\varepsilon(u_\varepsilon){\nabla} {u}_{\varepsilon}\cdot\nabla\psi }\\
&+\disp{(\varsigma-1)\chi\int_\Omega u_{\varepsilon}^{\varsigma-1}\nabla u_{\varepsilon}\cdot\nabla v_{\varepsilon}\psi+\chi\int_\Omega u_{\varepsilon}^{\varsigma} \nabla v_{\varepsilon}\cdot\nabla\psi}\\
&+\disp{(\varsigma-1)\xi\int_\Omega u_{\varepsilon}^{\varsigma-1}\nabla u_{\varepsilon}\cdot\nabla w_{\varepsilon}\psi+\xi\int_\Omega u_{\varepsilon}^{\varsigma} \nabla w_{\varepsilon}\cdot\nabla\psi}\\
&+\disp{\mu\int_\Omega u_{\varepsilon}^{\varsigma} \psi-\mu\int_\Omega u_{\varepsilon}^{\varsigma}w_{\varepsilon} \psi-\mu\int_\Omega u_{\varepsilon}^{\varsigma+1} \psi~~\mbox{for all}~~t\in(0,\infty).}\\
\end{array}
\label{cvbmdcfvgcz2.5ghju48}
\end{equation}
Next, we will estimate the right-hand sides of \dref{cvbmdcfvgcz2.5ghju48}.
To this end,
assuming that $p := \varsigma-m + 1$, then $\varsigma > m$ and $\varsigma\geq 2(m -1)$ yield to
$p > 1$ and $p\geq m- 1$. Since \dref{gbhnzjscz2.5297x9630111kkhhiioo}, we integrate \dref{cz2.5xx1jjjj} with respect to $t$ over $(0, T)$ for some fixed $T > 0$ and then have
  \begin{equation}
\begin{array}{rl}
&\disp{\frac{1}{{p}}\int_{\Omega}u_{\varepsilon}^{{{p}}}(\cdot,T)+
\frac{C_D(p-1)}{2}\int_{0}^T\int_{\Omega}u_{\varepsilon}^{m+p-3} |\nabla u_{\varepsilon}|^2}\\
 \leq&\disp{\frac{(p-1)\chi^2}{2C_D}\int_{0}^T\int_\Omega u_{\varepsilon}^{p+1-m}|\nabla v_{\varepsilon}|^2+\frac{({k-1})}{k}\xi\|w_0\|_{L^\infty(\Omega)}\int_\Omega  u_\varepsilon^{k}v_\varepsilon}\\
 &\disp{+(\mu+\kappa\xi)\int_{\Omega} u_\varepsilon^{k} +\mu
 C_1^pT+\frac{1}{{p}}\int_{\Omega}n_{0}^{{{p}}} }\\
 \leq&\disp{\frac{(p-1)\chi^2}{2C_D}C_1^{p+3-m}T +\mu
 C_1^pT+\frac{1}{{p}}\int_{\Omega}n_{0}^{{{p}}} .}\\
\end{array}
\label{vbhnjmkvgcz2.5ghhjuyuiihjj}
\end{equation}
 On the other hand, by $p = \varsigma-m + 1$, we have
  \begin{equation}
\int_{0}^T\int_{\Omega}u_{\varepsilon}^{\varsigma-2} |\nabla u_{\varepsilon}|^2=\int_{0}^T\int_{\Omega}u_{\varepsilon}^{m+p-3} |\nabla u_{\varepsilon}|^2\leq C_2(1+T)
 \label{fgbhvbhnjmkvgcz2.5ghhjuyuiihjj}
\end{equation}
for some positive constant $C_2$.
Next, by \dref{gbhnzjscz2.5297x9630111kkhhiioo}, we    also derive  that
 \begin{equation}
\begin{array}{rl}
&\disp{\mu\int_\Omega u_{\varepsilon}^{\varsigma} \psi-\mu\int_\Omega u_{\varepsilon}^{\varsigma}w_{\varepsilon} \psi-\mu\int_\Omega u_{\varepsilon}^{\varsigma+1}\psi \leq C_1^\varsigma|\Omega|(\mu+\mu\|w_0\|_{L^\infty(\Omega)}+\mu C_1)\|\psi\|_{L^\infty(\Omega)}}\\
\end{array}
\label{fbgnjvbhnjmkvgcz2.5ghhjuyuiihjj}
\end{equation}\
for all $\varepsilon\in(0, 1).$
Moreover, by \dref{gbhnzjscz2.5297x9630111kkhhiioo}--\dref{fbgnjvbhnjmkvgcz2.5ghhjuyuiihjj} and  the Young inequality, we conclude that there exists $C_3 > 0$ such that
 \begin{equation}
\begin{array}{rl}
\disp{|\int_{\Omega}\partial_{t}u^{\varsigma}_{\varepsilon}(\cdot,t)\cdot\psi|\leq C_3(
\int_{\Omega}u_{\varepsilon}^{\varsigma-2} |\nabla u_{\varepsilon}|^2+1)\|\psi\|_{W^{1,\infty}(\Omega)}.}\\
\end{array}
\label{fghncvbmdcfvgcz2.5ghju48}
\end{equation}
Due to the embedding $W^{2,q}(\Omega)\hookrightarrow W^{1,\infty}(\Omega)$ for $q > N$, we deduce 
that there exists
$C_4 > 0$ such that
\begin{equation}
\begin{array}{rl}
\disp{\|\partial_tu_{\varepsilon}^\varsigma(\cdot,t)\|_{(W^{2,q}(\Omega))^*}\leq C_4(
\int_{\Omega}u_{\varepsilon}^{\varsigma-2} |\nabla u_{\varepsilon}|^2+1)~~\mbox{for all}~~t\in(0,\infty)~~\mbox{and any}~~\varepsilon\in(0,1).}\\
\end{array}
\label{ggbhhfghncvbmdcfvgcz2.5ghju48}
\end{equation}
Now, combining   \dref{fgbhvbhnjmkvgcz2.5ghhjuyuiihjj} and \dref{ggbhhfghncvbmdcfvgcz2.5ghju48}, we can get \dref{zjscz2.5297x9630111kkhhiioott4}.
Now, observing that
the third equation of  \dref{5555710ssdd1.2x19x3189} is an ODE,  we derive that for any $(x,t)\in\Omega\times(0,\infty),$
\begin{equation}
\begin{array}{rl}
&\disp{w_\varepsilon(x,t)=w_0(x)e^{-\int_0^t v_\varepsilon(x,s)ds}.}\\
\end{array}
\label{vbgncz2.5xx1ffsskkoppgghh512}
\end{equation}
Hence,  by a basic
calculation, we conclude that for any $(x,t)\in\Omega\times(0,\infty),$
\begin{equation}
\begin{array}{rl}
&\disp{\nabla w_\varepsilon(x,t)=\nabla w_0(x)e^{-\int_0^t v_\varepsilon(x,s)ds}-w_0(x)e^{-\int_0^t v_\varepsilon(x,s)ds}\int_0^t  \nabla v_\varepsilon(x,s)ds,}\\
\end{array}
\label{vbgncz2.5xx1ffsskkopffggpgghh512}
\end{equation}
which together with \dref{x1.731426677gg} and \dref{gbhnzjscz2.5297x9630111kkhhiioo} implies \dref{zjsczsdddfff2.5297x9630111kkhhii}.
\end{proof}
\end{definition}
%

We are now in the position to prove our main result on global weak solvability.
\begin{lemma}\label{lemma45630223}
Assume that   $m> \frac{2N}{N+\gamma_*}$ with $N\geq1$,
 where $\gamma_*$ is given by \dref{xeerr1.731426677gg}.
 Then there exists $(\varepsilon_j)_{j\in \mathbb{N}}\subset (0, 1)$ such that $\varepsilon_j\rightarrow 0$ as $j\rightarrow\infty$ and that
\begin{equation} u_{\varepsilon}\rightarrow u ~~\mbox{a.e.}~~ \mbox{in}~~ \Omega\times (0,\infty),
\label{zjscz2.5297x9630222222}
\end{equation}
\begin{equation}
 u_{\varepsilon}\rightharpoonup u  ~~\mbox{weakly star in}~~ L^\infty(\Omega\times(0,\infty)),
 \label{zjscz2.5297x9630222222ee}
\end{equation}
\begin{equation}
v_\varepsilon\rightarrow v ~~\mbox{in}~~ C^0_{loc}(\bar{\Omega}\times[0,\infty)),
 \label{zjscz2.5297x96302222tt3}
\end{equation}
\begin{equation}
\nabla v_\varepsilon\rightarrow \nabla v ~~\mbox{in}~~ C^0_{loc}(\bar{\Omega}\times[0,\infty)),
 \label{zjscz2.5297x96302222tt4}
\end{equation}
\begin{equation}
\nabla v_\varepsilon\rightarrow \nabla v ~~\mbox{in}~~ L^{\infty}(\Omega\times(0,\infty))
 \label{zjscz2.5297x96302222tt4}
\end{equation}
as well as
\begin{equation}
 w_\varepsilon\rightharpoonup  w ~~\mbox{weakly star in}~~ L^{\infty}(\Omega\times(0,\infty))
 \label{zjscz2.5297x96302ssdd222tt4}
\end{equation}
and
\begin{equation}
\nabla w_\varepsilon\rightharpoonup \nabla w ~~\mbox{weakly star in}~~ L^{\infty}_{loc}(\Omega\times(0,\infty))
 \label{zjscz2.5297x96302ssdd22ssdd2tt4}
\end{equation}
 with some triple  $(u, v,w)$ which is a global weak solution of  \dref{5555710ssdd1.2x19x3189} in the sense of Definition \ref{df1}. 
\end{lemma}
\begin{proof}
Firstly, due to Lemma \ref{lemma45630hhuujjuu} and Lemma \ref{lemma45630hhuujjuuyytt}, 
for each $T > 0$,  we can find $\varepsilon$-independent constant $C(T)$ such that for all $t\in(0,T)$,
\begin{equation}
\|u_\varepsilon(\cdot,t)\|_{L^\infty(\Omega)}+ \|v_\varepsilon(\cdot,t)\|_{W^{1,\infty}(\Omega)} + \|w_\varepsilon(\cdot,t)\|_{W^{1,\infty}(\Omega)} \leq C(T)~~
\label{zjscz2.ssddd5297x9630111kk}
\end{equation}
as well as
\begin{equation}
\int_{0}^T\int_{\Omega}(u_{\varepsilon}+\varepsilon)^{m+p-3} |\nabla u_{\varepsilon}|^2\leq C(T)~~~\mbox{for any}~~p>1~~~\mbox{and}~~p\geq m- 1.
\label{fvgbhzjscz2.5sss297x96302222tt4455hyuhii}
 \end{equation}
  Now,   choosing $\varphi\in W^{1,2} (\Omega)$ as a second  function in the first equation in  \dref{5555710ssdd1.2x19x3189} and
using \dref{zjscz2.ssddd5297x9630111kk},
 we have
$$
\begin{array}{rl}
&\disp\left|\int_{\Omega}(v_{\varepsilon,t})\varphi\right|\\
 =&\disp{\left|\int_{\Omega}\left[\Delta v_{\varepsilon}-v_{\varepsilon}+u_{\varepsilon}\right]\varphi\right|}
\\
=&\disp{\left|\int_\Omega \left[-\nabla v_{\varepsilon}\cdot\nabla\varphi+v_{\varepsilon}\varphi+ u_{\varepsilon}  \varphi\right]\right|}\\
\leq&\disp{\left\{\|\nabla v_{\varepsilon}\|_{L^{2}(\Omega)}+ \|v_{\varepsilon}\|_{L^{2}(\Omega)}+ \|u_{\varepsilon}\|_{L^{2}(\Omega)}
\right\}}\times\disp{\|\varphi\|_{W^{1,2}(\Omega)}}\\
\end{array}
$$
for all $t>0$.
Along with \dref{zjscz2.ssddd5297x9630111kk}, further implies that
\begin{equation}
\begin{array}{rl}
&\disp\int_0^T\|\partial_tv_\varepsilon(\cdot,t)\|_{({W^{1,2}(\Omega)})^*}^{2}dt \\
\leq&\disp{\int_0^T\left\{\|\nabla u_{\varepsilon}\|_{L^{2}(\Omega)}+ \|v_{\varepsilon}\|_{L^{2}(\Omega)}+ \|n_{\varepsilon}\|_{L^{2}(\Omega)}
\right\}}^{2}dt
\\
\leq&\disp{C_1(T),}\\
\end{array}
\label{gbhncvbmdcfvgczffghhh2.5ghju48}
\end{equation}
where $C_1$ is a  positive constant independent of $\varepsilon$.
 Combining estimates \dref{zjscz2.ssddd5297x9630111kk}--\dref{gbhncvbmdcfvgczffghhh2.5ghju48} and the fact that $w_\varepsilon \leq \|w_0\|_{L^\infty(\Omega)}$, we can pick a sequence $(\varepsilon_j)_{j\in \mathbb{N}}\subset (0, 1)$ with $\varepsilon=\varepsilon_j\searrow0$ as $j\rightarrow\infty$  such that \dref{zjscz2.5297x9630222222ee}--\dref{zjscz2.5297x96302ssdd22ssdd2tt4} are valid with certain limit
functions $u,v$ and $w$ belonging to the indicated spaces.
%
We next fix $\varsigma> m$ satisfying $\varsigma\geq2(m-1)$ and set $p := 2\zeta-m+1$, then by \dref{fvgbhzjscz2.5sss297x96302222tt4455hyuhii} implies
that
for each
$T > 0,$ $(u_{\varepsilon}^\varsigma)_{\varepsilon\in(0,1)}$ is bounded in $L^2((0, T);W^{1,2}(\Omega))$.
 With the help of Lemma \ref{lemma45630hhuujjuuyytt}, we  also show that
$$(\partial_{t}u_{\varepsilon}^\varsigma)_{\varepsilon\in(0,1)}~~\mbox{is bounded in}~~L^1((0, T); (W^{2,q}(\Omega))^*)~~\mbox{for each}~~  T > 0$$
and some $q>N.$
Hence, an Aubin-Lions lemma (see e.g. \cite{Simon})
applies to the above inequality we have the strong precompactness of
%
%
%
$(u_{\varepsilon}^\varsigma)_{\varepsilon\in(0,1)}$ in $L^2(\Omega\times(0, T))$.
Therefore,
we can
  pick a
suitable subsequence such that $u_{\varepsilon}^\varsigma\rightarrow z^\varsigma$ for some nonnegative measurable $z:\Omega\times(0,\Omega)\rightarrow\mathbb{R}$.
 In
light of  \dref{zjscz2.5297x9630222222ee} and the Egorov theorem, we have $z = u$ necessarily, so that \dref{zjscz2.5297x9630222222} is valid.

 Next we shall prove that $(u,v,w)$ is a weak solution of problem \dref{7101.2x19x3sss189}.  To this end, multiplying the first equation
 as well as  the
second  equation and third equation in  \dref{5555710ssdd1.2x19x3189} by $\varphi\in C^\infty_0(\Omega\times [0,\infty))$, we obtain
\begin{equation}
\begin{array}{rl}\label{eqx4ss5xx12112ccgghh}
\disp{-\int_0^{\infty}\int_{\Omega}u_\varepsilon\varphi_t-\int_{\Omega}u_0\varphi(\cdot,0)  }=&\disp{
\int_0^\infty\int_{\Omega}H(u_\varepsilon+\varepsilon)ds\Delta\varphi+\chi\int_0^\infty\int_{\Omega}u_\varepsilon\nabla v_\varepsilon\cdot\nabla\varphi}\\
&+\disp{\xi\int_0^\infty\int_{\Omega}u_\varepsilon\nabla w_\varepsilon\cdot\nabla\varphi+\int_0^{\infty}\int_{\Omega}(\mu u_\varepsilon-\mu u_\varepsilon w_\varepsilon- \mu u_\varepsilon^2)\varphi}\\
\end{array}
\end{equation}
as well as
\begin{equation}
\begin{array}{rl}\label{eqx45xsssx12112ccgghhjj}
&\disp{-\int_0^{\infty}\int_{\Omega}v_\varepsilon\varphi_t-\int_{\Omega}v_0\varphi(\cdot,0)=-
\int_0^\infty\int_{\Omega}\nabla v_\varepsilon\cdot\nabla\varphi-\int_0^\infty\int_{\Omega}v_\varepsilon\varphi+\int_0^\infty\int_{\Omega}u_\varepsilon\varphi}
\end{array}
\end{equation}
and
\begin{equation}
\begin{array}{rl}\label{sddddfgghh}
&\disp{-\int_0^{\infty}\int_{\Omega}w_\varepsilon\varphi_t-\int_{\Omega}w_0\varphi(\cdot,0)=-
\int_0^\infty\int_{\Omega}v_\varepsilon w_\varepsilon\varphi.}
\end{array}
\end{equation}
 for all $\varepsilon\in (0,1)$, where $H$ is given by \dref{zjscz2.5297x963ssddd0111kkuu}.
 Then \dref{zjscz2.5297x9630222222}--\dref{zjscz2.5297x96302222tt4},  and the dominated convergence theorem enables
us to conclude
$$
\begin{array}{rl}\label{eqx45xx12112ccgghh}
&\disp{-\int_0^{\infty}\int_{\Omega}u\varphi_t-\int_{\Omega}u_0\varphi(\cdot,0) }
\\
=&\disp{
\int_0^\infty\int_{\Omega}H(u)\Delta\varphi+\chi\int_0^\infty\int_{\Omega}u
\nabla v\cdot\nabla\varphi+\xi\int_0^\infty\int_{\Omega}u
\nabla w\cdot\nabla\varphi}
+\disp{\int_0^\infty\int_{\Omega}(\mu u-\mu uw-\mu u^2)\varphi}
\end{array}
$$
as well as
$$
\begin{array}{rl}\label{eqx45xx12112ccgghhjj}
&\disp{-\int_0^{\infty}\int_{\Omega}v\varphi_t-\int_{\Omega}v_0\varphi(\cdot,0)=-
\int_0^\infty\int_{\Omega}\nabla v\cdot\nabla\varphi-\int_0^\infty\int_{\Omega}v\varphi+\int_0^\infty\int_{\Omega}u\varphi}
\end{array}
$$
and
$$
\begin{array}{rl}
&\disp{-\int_0^{\infty}\int_{\Omega}w\varphi_t-\int_{\Omega}w_0\varphi(\cdot,0)=-
\int_0^\infty\int_{\Omega}v w\varphi}
\end{array}
$$
by a limit procedure.
The proof of Lemma \ref{lemma45630223} is completed.
\end{proof}

%

%
%

We can now easily prove our main result.

{\bf The proof of Theorem \ref{theossdrem3}}~

A combination of Lemma \ref{lemma45630hhuujjuu} and Lemma \ref{lemma45630223} directly leads to
our desired result.


{\bf Acknowledgement}:
 This work is partially supported by the National Natural Science
Foundation of China (No. 11601215), Shandong Provincial Science Foundation for Out-
standing Youth (No. ZR2018JL005) and Project funded by China Postdoctoral Science
Foundation (No. 2019M650927, 2019T120168).

%
%

%


\begin{thebibliography}{00}






%


%













 \bibitem{Alikakos72}  N. D. Alikakos, \textit{$L^p$ bounds of solutions of reaction-diffusion equations},  Comm. Partial Diff. Eqns., 4(1979), 827--868.



\bibitem{Bellomo} N. Bellomo, N. K. Li, P. K.Maini, \textit{On the foundations of cancer modelling: Selected topics,
speculations, and perspectives}, Math. Models Methods Appl. Sci., 18(2008),  593--646.


\bibitem{Bellomo1216} N. Bellomo,  A. Belloquid,   Y. Tao, M. Winkler,  \textit{Toward a mathematical theory of
Keller--Segel models of pattern formation in biological tissues}, Math. Models Methods Appl. Sci., 25(2015), 1663--1763.












%
%


\bibitem{Calvez710} V. Calvez,   J. A. Carrillo, \textit{Volume effects in the Keller--Segel model: Energy estimates
preventing blow-up,} J. Math. Pures Appl., 86(2006),    155--175.







%
%
\bibitem{Cao} X. Cao, \textit{Boundedness in a three-dimensional chemotaxis--haptotaxis model}, Z. Angew. Math. Phys., 67(2016), 1--13.




 \bibitem{Chaplain3} M. A. J. Chaplain, A. R. A. Anderson, \textit{Mathematical modelling of tissue invasion, in
Cancer Modelling and Simulation}, L. Preziosi, ed., Chapman  Hall/CRC, Boca Raton,
FL, 2003,  267--297.



 \bibitem{Chaplain1}M. A. J. Chaplain, G. Lolas, \textit{Mathematical modelling of cancer invasion of tissue: The
role of the urokinase plasminogen activation system}, Math. Models Methods Appl. Sci., 11(2005),  1685--1734.



 \bibitem{Chaplain7}  M. A. J. Chaplain, G. Lolas, \textit{Mathematical modelling of cancer invasion of tissue: dynamic heterogeneity},
Net. Hetero. Med., 1(2006), 399--439.


%




%
%

\bibitem{Cie72}T. Cie\'{s}lak, M. Winkler, \textit{Finite-time blow-up in a quasilinear system of chemotaxis,}
Nonlinearity, 21(2008), 1057--1076.


%







%
%

%
%


%
%
%


 \bibitem{Friedman}A. Friedman, G. Lolas, \textit{Analysis of a mathematical model of tumor lymphangiogenesis},
Math. Models Methods Appl. Sci., 15(2005), 95--107.

%
%


\bibitem{Hajaiej} H. Hajaiej, L. Molinet, T. Ozawa, B. Wang, \textit{Necessary and sufficient conditions for the fractional Gagliardo--Nirenberg inequalities and applications to Navier--Stokes and generalized boson equations}, in: Harmonic Analysis
and Nonlinear Partial Differential Equations, in: RIMS K\^{o}ky\^{u}roku Bessatsu, vol. B26, Res. Inst. Math. Sci. (RIMS),
Kyoto, 2011, pp. 159--175.


\bibitem{Haroske} D. D. Haroske, H. Triebel, \textit{Distributions, Sobolev Spaces, Elliptic Equations, European Mathematical Society,} Zurich, 2008.








%
%
%
%
%
%
\bibitem{Herrero710} M. Herrero,   J.  Vel\'{a}zquez, \textit{A blow-up mechanism for a chemotaxis model,} Ann.
Scuola Norm. Super. Pisa Cl. Sci., 24(1997),    633--683.


\bibitem{Hieber}  M. Hieber, J. Pr\"{u}ss, \textit{Heat kernels and maximal $L^p$-$L^q$ estimate for parabolic evolution equations}, Comm. Partial Diff. Eqns., 22(1997), 1647--16
%

\bibitem{Hillen79} T. Hillen,  K. J. Painter,  \textit{A use's guide to PDE models for chemotaxis,} J. Math. Biol., 58(2009), 183--217.

\bibitem{Hillen5662710} T. Hillen, K. J. Painter, \textit{Global existence for a parabolic chemotaxis model with
prevention of overcrowding}, Adv. Appl. Math., 26(2001), 281--301.



\bibitem{Hillensxdc79}  T. Hillen, K. J. Painter and M. Winkler, \textit{Convergence of a cancer invasion model to a logistic
chemotaxis model}, Math. Models Methods Appl. Sci., 23(2013), 165--198.


%

\bibitem{Horstmann2710}D. Horstmann,  \textit{From $1970$ until present: the Keller--Segel model in chemotaxis and its consequences,} I.
Jahresberichte der Deutschen Mathematiker-Vereinigung, 105(2003), 103--165.



\bibitem{Horstmann791} D. Horstmann, M. Winkler, \textit{Boundedness vs. blow-up in a chemotaxis system}, J. Diff. Eqns, 215(2005), 52--107.

\bibitem{HuHuHueeezseeddd0} X. Hu, L. Wang, C. Mu, et al.,  \textit{Boundedness in a three-dimensional chemotaxis-haptotaxis model with nonlinear diffusion}, Comptes Rendus Mathematique, 355(2017), 181--186.


\bibitem{Ishida}S. Ishida, K. Seki, T, Yokota, \textit{Boundedness in quasilinear Keller--Segel systems of parabolic--parabolic type on
non-convex bounded domains}, J. Diff. Eqns., 256(2014), 2993--3010.



\bibitem{Jger317} W. J\"{a}ger, S. Luckhaus,  \textit{On explosions of solutions to a system of partial differential equations modelling
chemotaxis}, Trans. Am. Math. Soc., 329(1992), 819--824.


 \bibitem{Jineeezseeddd0}  C. Jin,  \textit{Boundedness and global solvability to a chemotaxis-haptotaxis model with slow and fast diffusion}, Discrete  Continuous Dynamical Systems, 23(2018), 1675--1688.


 \bibitem{Keller79} E. Keller, L. Segel, \textit{Initiation of slime mold aggregation viewed as an instability, }  J. Theor. Biol., 26(1970), 399--415.




\bibitem{Keller791} E. Keller, L. Segel, \textit{Traveling bands of chemotactic bacteria: a theoretical analysis,} J. Theor. Biol., 30(1971), 377--380.




\bibitem{Kowalczyk7101}R. Kowalczyk, \textit{Preventing blow-up in a chemotaxis model}, J. Math. Anal. Appl., 305(2005),
   566--585.

\bibitem{LiLiLi791} Y. Li, J. Lankeit, \textit{Boundedness in a chemotaxis-haptotaxis model with nonlinear diffusion}, Nonlinearity, 29(2016), 1564--1595.












%

%
%
%
%
%
%
%
%

\bibitem{LiottaLiotta3710} L. A. Liotta,  T. Clair,   \textit{Cancer: checkpoint for invasion}, Nature, 405(2000), 287--288.


\bibitem{Lianu} G. Li\c{t}anu, C. Morales-Rodrigo,  \textit{Global solutions and asymptotic behavior for a
parabolic degenerate coupled system arising from biology}, Nonlinear Anal., 72(2010),
77--98.

\bibitem{Lianu1} G. Li\c{t}anu, C. Morales-Rodrigo, \textit{Asymptotic behavior of global solutions to a model of
cell invasion}, Math. Models Methods Appl. Sci., 20(2010),  1721--1758.

\bibitem{Liughjj791} J. Liu, J. Zheng,  Yifu Wang, \textit{Boundedness in a quasilinear chemotaxis-haptotaxis system with logistic source},
Z. Angew. Math. Phys., 67(2016), 1--33.





\bibitem{Marciniak} A. Marciniak-Czochra, M. Ptashnyk,  \textit{Boundedness of solutions of a haptotaxis model},
Math. Models Methods Appl. Sci., 20(2010),  449--476.

\bibitem{Nagaixcdf791} T. Nagai, T. Senba,K. Yoshida   \textit{Application of the Trudinger-Moser inequality to a parabolic system of
chemotaxis}, Funkc. Ekvacioj, Ser. Int.,  40(1997), 411--433.









%




%
%

\bibitem{Osakix391} K. Osaki, T. Tsujikawa, A. Yagi,   M. Mimura, \textit{Exponential attractor for a chemotaxisgrowth
system of equations}, Nonlinear Anal. TMA., 51(2002),  119--144.


\bibitem{Perthame317} B. Perthame, \textit{Transport Equations in Biology}, Birkh\"{a}user Verlag, Basel, Switzerland, 2007.


\bibitem{Simon} J. Simon, \textit{Compact sets in the space $L^{p}(O, T;B)$}, Annali di Matematica Pura ed Applicata, 146(1986), 65--96.





\bibitem{Stinnerddff12}  C. Stinner, C. Surulescu, M. Winkler, \textit{Global weak solutions in a PDE-ODE system modeling
multiscale cancer cell invasion}, SIAM J. Math. Anal., 46(2014), 1969--2007.



%





%
%
%










\bibitem{Sugiyama710}Y. Sugiyama, \textit{Time global existence and asymptotic behavior of solutions to degenerate quasilinear
parabolic systems of chemotaxis}, Diff. Integral Eqns., 20(2007),    133--180.

\bibitem{Szymaska} Z. Szymaska, C. Morales-Rodrigo, M. Lachowicz,  M. Chaplain, \textit{Mathematical modelling
of cancer invasion of tissue: The role and effect of nonlocal interactions}, Math.
Models Methods Appl. Sci., 19(2009),  257--281.


\bibitem{Taossddssd793} Y. Tao,   \textit{Global existence for a haptotaxis model of cancer invasion with tissue remodeling},
Nonlinear Anal. RWA., 12(2011), 418--435.





\bibitem{Taox3201}Y. Tao, \textit{Boundedness in a two-dimensional chemotaxis--haptotaxis system}, Mathematics,  70(2014), 165--174.


\bibitem{Tao3} Y. Tao, M. Wang,  \textit{Global solution for a chemotactic--haptotactic model of cancer invasion},
Nonlinearity, 21(2008), 2221--2238.


\bibitem{Tao2}Y. Tao, M. Wang, \textit{A combined chemotaxis--haptotaxis system: The role of logistic source},
SIAM J. Math. Anal., 41(2009),  1533--1558.



%

\bibitem{Tao72} Y. Tao, M. Winkler,  \textit{A chemotaxis--haptotaxis model: the roles of porous medium
diffusion and logistic source}, SIAM J. Math. Anal., 43(2011), 685--704.


\bibitem{Tao794} Y. Tao, M. Winkler,  \textit{Boundedness in a quasilinear parabolic--parabolic Keller--Segel system with subcritical sensitivity}, J.
Diff. Eqns., 252(2012), 692--715.





%



\bibitem{Taox26} Y. Tao, M. Winkler,
\textit{Boundedness and stabilization in
a multi-dimensional chemotaxis--haptotaxis model},
Proceedings of the Royal Society of Edinburgh, 144(2014), 1067--1084.

\bibitem{Taox26216} Y. Tao, M. Winkler, \textit{Dominance of chemotaxis in a
chemotaxis--haptotaxis model},
Nonlinearity, 27(2014), 1225--1239.

\bibitem{Tao79477} Y. Tao, M. Winkler,  \textit{Energy-type estimates and global solvability in a two-dimensional chemotaxis--haptotaxis
model with remodeling of non-diffusible attractant}, J. Diff. Eqns.,  257(2014), 784--815.

\bibitem{Tao79477ddffvg} Y. Tao, M. Winkler,  \textit{Large time behavior in a multidimensional chemotaxis--haptotaxis model with slow signal diffusion},
SIAM J. Math. Anal., 47(2015), 4229--4250.




\bibitem{Tello710} J. I. Tello,   M. Winkler, \textit{A chemotaxis system with logistic source}, Comm. Partial
Diff. Eqns., 32(2007),    849--877.

%


%




\bibitem{Walker} C. Walker, G. F. Webb, \textit{Global existence of classical solutions for a haptotaxis model},
SIAM J. Math. Anal., 38(2007),  1694--1713.

%

\bibitem{Wangscd331629} Y.  Wang,  \textit{Boundedness in the higher-dimensional chemotaxis-haptotaxis model with nonlinear diffusion},
J. Diff. Eqns., 260(2016), 1975--1989.



\bibitem{Winkler79} M. Winkler, \textit{Does a volume-filling effect always prevent chemotactic collapse}, Math. Methods Appl. Sci., 33(2010), 12--24.


 \bibitem{Winkler37103}M. Winkler, \textit{Boundedness in the higher-dimensional parabolic--parabolic chemotaxis system with
logistic source}, Comm.  Partial Diff. Eqns., 35(2010), 1516--1537.



 \bibitem{Winkler792} M. Winkler, \textit{Aggregation vs. global diffusive behavior in the higher-dimensional Keller--Segel model}, J. Diff.
Eqns., 248(2010), 2889--2905.


\bibitem{Winkler715} M. Winkler,  \textit{Blow-up in a higher-dimensional chemotaxis system despite logistic
growth restriction,}  J. Math. Anal. Appl., 384(2011), 261--272.



%

%
\bibitem{Winkler793} M. Winkler, \textit{Finite-time blow-up in the higher-dimensional parabolic--parabolic Keller--Segel system}, J. Math. Pures
Appl., 100(2013),  748--767.



%
%




\bibitem{Zheng} J. Zheng, \textit{Boundedness of solutions to a quasilinear parabolic--elliptic Keller--Segel system with logistic source},
J. Diff. Eqns., 259(2015), 120--140.



 \bibitem{Zhddengssdeeezseeddd0} J. Zheng, \textit{Boundedness of solutions to a quasilinear higher-dimensional chemotaxis--haptotaxis model with nonlinear diffusion},
Discrete and Continuous Dynamical Systems,  37(2017), 627--643.


      \bibitem{Zhensddfffsdddgssdddssddddkkllssssssssdefr23} J. Zheng, \textit{A note on boundedness of solutions to a higher-dimensional quasi-linear chemotaxis system with logistic source}, Zeitschriftf\"{u}r Angewandte Mathematik und Mechanik, 97(2017), 414--421.







\bibitem{Zhenddddgssddsddfff00} J. Zheng, \textit{An optimal result for global existence and boundedness in a three-dimensional
 Keller-Segel-Stokes system   with  nonlinear diffusion},  J. Diff. Eqns., 267(2019), 2385--2415.


    \bibitem{Zhengssdddsssddfghhhddddkkllssssssssdefr23}    J. Zheng, \textit{A new result for global classical solution and boundedness to a
chemotaxis-haptotaxis model with re-establishment mechanisms}, Preprint.


 \bibitem{Zhengssdddssddddkkllssssssssdefr23} J. Zheng et. al., \textit{A new result for global existence and boundedness of
  solutions to a parabolic--parabolic Keller--Segel system with logistic source}, J. Math. Anal. Appl., 462(2018), 1--25.



           \bibitem{zhengjjkk}P. Zheng, C. Mu, X. Song, et al., \textit{On the boundedness and decay of solutions for a chemotaxis-haptotaxis system with nonlinear diffusion},
                Discrete and Continuous Dynamical Systems, 36(2015), 1737--1757.



\end{thebibliography}
\end{document}